\documentclass[11pt,english]{article}
\usepackage[utf8]{inputenc}
\usepackage[T1]{fontenc}
\usepackage{babel}
 \usepackage{amsmath}
 \usepackage{mathtools}
\usepackage{amsfonts}
\usepackage{amssymb}
\usepackage{graphicx}
\usepackage{amsthm}
\usepackage{bbm}
\usepackage{mathtools}
\usepackage[colorlinks=true,
            linkcolor=blue,
            urlcolor=blue,
            citecolor=blue]{hyperref}
\usepackage{authblk}
\usepackage{geometry}
\usepackage{algorithm}
\usepackage[noend]{algpseudocode}
\usepackage[vcentermath]{youngtab}
\usepackage{pgfplots}
\pgfplotsset{compat=newest}
\usetikzlibrary{calc}
\usepackage{mathrsfs}
\DeclarePairedDelimiter\ceil{\lceil}{\rceil}

 \author{Mohamed Slim Kammoun\thanks{Partially supported  by the Labex CEMPI ANR-11-LABX-0007-01.} 
\thanks{Partially supported by a Leverhulme Trust Research Project Grant RPG-2020-103.}\\
\small Department of Mathematics and Statistics\\
\small Lancaster University\\ 
\small Lancaster, U.K.\\
\small\tt m.kammoun@lancaster.ac.uk\\
}

\title{Universality for random permutations and some other groups}
\newtheorem{theorem}{Theorem}
\newtheorem{corollary}[theorem]{Corollary}
\newtheorem{lemma}[theorem]{Lemma}
\newtheorem{proposition}[theorem]{Proposition}

\theoremstyle{definition}
\newtheorem{definition}[theorem]{Definition}
\newtheorem{remark}[theorem]{Remark}

\usepackage{tikz}
\usetikzlibrary{automata, positioning}
\usepackage{natbib}
\newcommand{\p}{\mathbb{P}}
\newcommand{\E}{\mathbb{E}}

\newcommand{\s}{\mathfrak{S}_n}
\definecolor {processblue}{cmyk}{0.96,0,0,0}

\begin{document}
\maketitle
\begin{abstract} 
We present some Markovian   approaches to prove  universality results for some functions on the symmetric group. 
Some of those statistics are already studied in  \citep{kammoun2018,kam2} but not the general case. We  prove, in particular, that  the number of occurrences of a vincular  patterns satisfies a  CLT for conjugation invariant random permutations with few cycles and we improve the results already known for the longest increasing subsequence.
The second approach is a suggestion of a  generalization to other random permutations and other sets having a similar structure than the symmetric group. 
\end{abstract}

\section{The ping-pong method } \label{sec:1}
Let $\s$ be the group of permutations of $[n]$. For $\sigma\in\s$, we denote by $\#(\sigma)$ the number of cycle of $\sigma$. Now let
$$\s^0:=\{\sigma\in\s: \#(\sigma)=1\} .$$ 
In this section, we are interested in proving universality for conjugation invariant random permutations with few cycles. 
 A sequence of random permutations $(\sigma_n)_{n\geq 1}$ is said to be conjugation invariant if $\sigma_n$ is supported on $\s$ and 
 \begin{align}
 \tag{$\mathcal{H}_{inv}$}\label{hinv}
\forall n\geq 1, \, \forall \sigma\in\s,\,  \sigma_n \overset{d}=  \sigma^{-1}\sigma_n\sigma.   
 \end{align} 
 
For $\alpha \geq 1$ and $p\in[1,\infty]$, we say that the sequences of random permutations $(\sigma_n)_{n\geq 1}$ satisfies  $\mathcal{H}_{inv,\alpha}^\mathbb{P}$  if 
\begin{align} \label{hinv1}
 (\sigma_n)_{n\geq 1} \text { is conjugation invariant and } \quad 
\tag{$\mathcal{H}_{inv,\alpha}^\mathbb{P}$}
 \frac{\#(\sigma_n)}{n^ {\frac1\alpha}} \xrightarrow[n\to\infty]{\p}0,
\end{align}
and we say that it satisfies 
$\mathcal{H}_{inv,\alpha}^{\mathbb{L}^p}$ if 
\begin{align} \label{hinv1p}
 (\sigma_n)_{n\geq 1} \text { is conjugation invariant and } \quad 
\tag{$\mathcal{H}_{inv,\alpha}^{\mathbb{L}^p}$}
 \frac{\#(\sigma_n)}{n^ {\frac1\alpha}} \xrightarrow[n\to\infty]{\mathbb{L}^p}0.
\end{align}


\subsection[Rebound on the Ewens zero distribution ]{Rebound on the Ewens zero distribution }  
\paragraph{}
Given $n\geq1$ and $E \subset \s$, we define 
$$ \mathrm {next} (E) : = \{\rho \circ (i, j); \rho \in E, \, \,
\# (\rho \circ (i, j)) = \# (\rho) -1 \} \cup \{\rho \in E; \# (\rho) = 1 \} $$
and $$ \mathrm{final}(\sigma):=\begin{cases} 
\mathrm{next}^{\#(\sigma)-1}( \{\sigma\})  & \text{if }  \#(\sigma)>1   \\ \{\sigma\} & \text{otherwise} \end{cases}.$$ 
In other words, $\mathrm{next}(E)$ is the set of permutations obtained by concatenating, if possible, two cycles of some
$\sigma \in E$, and $\mathrm{final}(\sigma)$ is the set of permutations obtained by concatenating all the cycles of $\sigma$. In particular, \label{def:sigma0} $$\mathrm{final}(\sigma)\subset\s^0:=\{\sigma\in\s;\,\#(\sigma)=1\}.$$   
\paragraph{}
Let ${\mathcal{G}_{\mathfrak{S}_n}}$ be the directed graph  with vertices $\s$ and  edges $\{(\sigma,\rho) ; \sigma\in \s, \rho \in \mathrm{next}(\{\sigma\}) \}$.
We represent ${\mathcal{G}_{\mathfrak{S}_3}}$ in Figure~\ref{figtr}. ${\mathcal{G}_{\mathfrak{S}_n}}$ can be seen as a directed version of the Cayley graph of $\s$ generated by transpositions where the edges are oriented toward the permutations with fewer cycles (the further from the identity according to the graph distance), for which we added loops at the permutations of $\s^0$.    In this first part of this  section, we will examine the uniform random walk on ${\mathcal{G}_{\mathfrak{S}_n}}$.

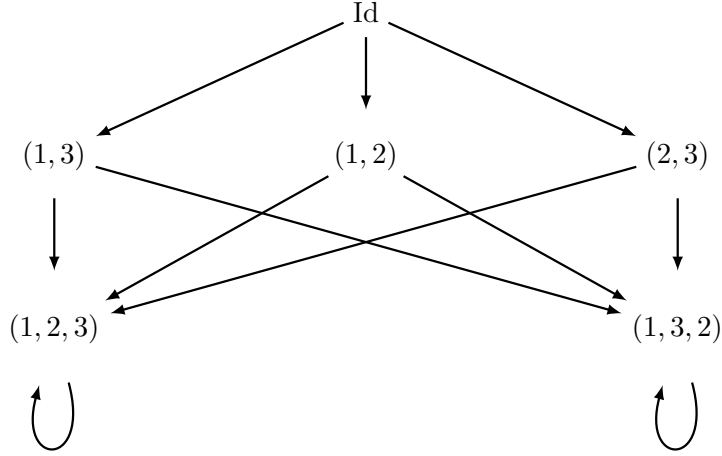
\begin{figure}
\centering
\begin{tikzpicture}
[-latex,auto , 
semithick, scale=1, every node/.style={transform shape}, state/.style={circle,inner sep=2pt}]
    \node[state] (s1)  {Id};
    \node[state, below=1cm of s1] (t1) {$(1,2)$};
    \node[state, right=3cm of t1] (t2) {$(2,3)$};
    \node[state, left=3cm of t1] (t3) {$(1,3)$};
    \node[state, below=1cm  of t3] (c1) {$(1,2,3)$};
    \node[state, below=1cm of t2] (c2) {$(1,3,2)$};
        \draw[every loop,
        line width=0.3mm,
        auto=left,
        >=latex,
        ]
            (s1) edge[]  node  {} (t1)
             (s1) edge[]  node {} (t3)
              (s1) edge[]  node {} (t2)
                 (t1) edge[ ]  node {} (c1)
       (t2) edge[ ]  node {} (c1)
                 (t3) edge[ ]  node {} (c1)
                 (t1) edge[]  node {} (c2)
                 (t2) edge[]  node {} (c2)
                 (t3) edge[]  node {} (c2)
                 (c2) edge[loop below]  node {} (c2)
                 (c1) edge[loop below]  node {} (c1);
 \end{tikzpicture}
    \caption{The directed graph ${\mathcal{G}_{\mathfrak{S}_3}}$}
    \label{figtr}
\end{figure}
\paragraph{}
Let   $f$ be a function defined on \label{def:sigmainf} $\mathfrak{S}_{\infty}:=\cup_{i=1}^\infty \mathfrak{S}_n$ and taking its values in some metric space  
  $(F,d_F)$, for example $\mathbb{Z}$, $\mathbb{R}$, \label{def:ens} $\mathbb{R}^d$  or $\mathscr{C}^0(\mathbb{R})$.  
  It turns out that the uniform distribution on $\s^0$, also known as the  Ewens distribution\footnote{See Appendix \ref{apn1} for more details.} with parameter~$0$,
  is useful to obtain universality results for conjugation invariant permutations  if $f$ does not change too much by merging two cycles. 
  More precisely, we define for $1\leq k\leq  n$,
\begin{align*} 
    \varepsilon'_{n,k}(f)&:=\max_{\sigma\in\s, \#(\sigma)=k} \max_{\rho\in \mathrm{final}(\sigma)} d_F(f(\sigma),f(\rho)).
    \end{align*}
    We present now our main result. 
\begin{theorem} \label{main_rest1} Assume that $(\sigma_n)_{n\geq 1}$ and $(\sigma_{ref,n})_{n\geq 1}$  satisfy  \eqref{hinv}. 
Suppose that there exists   $x\in F$   such that
\begin{align} \label{faddddit}
f(\sigma_{ref,n}) &\xrightarrow[n\to\infty]{\mathbb{P}} x, 
\\  \label{control_erreur}  
\varepsilon'_{n,\#(\sigma_{ref,n})}(f) &\xrightarrow[n\to\infty]{\mathbb{P}} 0  \quad 
\\   \label{control_erreur2}
\text{and that} \quad \varepsilon'_{n,\#(\sigma_n)}(f) &\xrightarrow[n\to\infty]{\mathbb{P}} 0.
\quad 
    \end{align}
Then
\begin{align}\label{conv1}
f(\sigma_{n}) \xrightarrow[n\to\infty]{\mathbb{P}} x.  
\end{align}
Moreover, if the assumptions \eqref{faddddit}--\eqref{control_erreur2} hold true for the $\mathbb{L}^p$ convergence for some $p\geq 1$ instead of the convergence in probability, then so does \eqref{conv1}.
\end{theorem}

When $F=\mathbb{R}^d$, we obtain also the convergence in distribution.
\begin{theorem} \label{universality_RD}
Assume that $F=\mathbb{R}^d$ and that  $(\sigma_n)_{n\geq 1}$ and   $(\sigma_{ref,n})_{n\geq 1}$ satisfy \eqref{hinv}.
Suppose that     \eqref{control_erreur} and \eqref{control_erreur2} hold true  and that there exists a random variable   $X$ supported on  $F$   such that
\begin{align*}
    f(\sigma_{ref,n}) \xrightarrow[n\to\infty]{d} X.
\end{align*}
Then \label{def:covd}
\begin{align} \label{conv2}
    f(\sigma_{n}) \xrightarrow[n\to\infty]{d} X. 
\end{align}
\end{theorem}
Let $\sigma_{unif,n}$ and 
$\sigma_{Ew,0,n}$ be uniform random permutations respectively on $\s$ and $\s^0$.  
The idea of the proof is to compare both  $f(\sigma_n)$ and $f(\sigma_{ref,n}$) with $f(\sigma_{Ew,0,n})$. In general, the choice  
$\sigma_{ref,n}\overset{d}{=}\sigma_{unif,n}$ 
is interesting since,  the convergence in \eqref{faddddit} is known for many statistics. Moreover,  using Proposition~\ref{numb_cyc_unif}, we have immediately the following result. 
 \begin{corollary} If $\sigma_{ref,n}\overset d {=}\sigma_{unif,n}$,  in both theorems~\ref{main_rest1} and \ref{universality_RD},  the hypothesis \eqref{control_erreur} can be replaced  by the existence of  $\kappa>0$ such that 
\begin{align*}  
    \max_{ \left|\frac{k}{\log(n)}-1\right|<\kappa} \varepsilon'_{n,k}(f) &\xrightarrow[n\to\infty]{\mathbb{P}} 0.
\end{align*}
\end{corollary}
We chose to give a very simple version that can be checked easily for  many statistics.  For almost sure convergence, one can obtain similar results after defining properly  the spaces. We will not discuss here this type of convergence.  We will give many applications using the following observation. 
 \begin{remark} \label{remarque_RD}
By the  triangle inequality, we have
$$\varepsilon'_{n,k}(f)\leq \sum_{i=2}^{k}\varepsilon_{n,i}(f) \leq (k-1)\varepsilon_n(f), $$ 
where 
$$         \varepsilon_{n,k}(f):=\max_{\sigma\in\s, \#(\sigma)=k} \max_{\rho\in \mathrm{next}(\{\sigma\})} d_F(f(\sigma),f(\rho))
    \quad \text{and} \quad \varepsilon_n(f):= \max_{1\leq k<n} \varepsilon_{k,n}(f).$$
Consequently, if there exists some  $\alpha\leq 1$ such that 
$$\varepsilon_n(f)=O\left(\frac 1{n^\frac 1\alpha}\right)$$
then \eqref{hinv1} implies \eqref{control_erreur2} and \eqref{hinv1p} implies the equivalent hypothesis in $\mathbb{L}^p$.   Moreover, if \linebreak $\sigma_{ref,n}\overset{d}=\sigma_{unif,n}$, then 
Proposition~\ref{lpewens} implies \eqref{control_erreur}.  We will give some direct applications of this observation in the next subsection. 
 \end{remark}

\subsection{Some applications}
In the next corollary, we will give some applications. 
The first column of  Table~\ref{tab:my-table} contains the function to study. We apologize to the reader because  those statistics are not defined yet. One can check the corresponding result in the fifth column for more details. 

\begin{corollary}\label{first_bound_cor}
For the functions $f$ the distribution $X$ and the real $\alpha$ in Table~\ref{tab:my-table}, if \eqref{hinv1}   is satisfied, then 
\begin{align*}
    f(\sigma_{n}) \xrightarrow[n\to\infty]{d} X 
\end{align*}
except for the sixth example where  the convergence holds in probability.\footnote{In the space of continual diagrams i.e. the set of  $1$-Lipschitz real functions $f$ such that outside one compact, $f(x)=|x-a|$. One can see \citep{ss1,MR3733197} for more details for continual diagrams. We will use as distance, $d_F(f,g)=\sup_{x\in \mathbb{R}} |f(x)-g(x)| $ which is finite since both functions are continuous and outside one  compact of $\mathbb{R}$,  $f-g$ is constant.}  For the first and the forth  examples the convergence holds also in $\mathbb{L}^p$ under \eqref{hinv1p}. For the fifth example please check the corresponding theorem for more details about the type of convergence. 
\end{corollary}

Note that: 
\begin{itemize}
 \item    We give in the third column the inequality  we used to obtain our results. Except for the cases where we study the RSK image of the permutation, the longest alternating subsequence and the descent process, the inequality  is trivial, but we will prove all the inequalities in the sequel.
 \item We want to emphasize that these results are just a direct application of theorems~\ref{main_rest1} and \ref{universality_RD}. Using more sophisticated controls of the error, one could obtain larger classes of universality as we will detail in the sequel.  
 \item For all our examples, the special case of the  Ewens distribution satisfies the hypothesis.
\end{itemize}
\begin{table}
\begin{tabular}{|l|l|l|l|l|}
\hline
$f(\sigma)$                                                         & X  
& Error                                     & Hypotheses & Theorem 
\\ \hline
$\frac{\mathrm{LIS}(\sigma)}{\sqrt{n}}$, $\frac{\mathrm{LDS}(\sigma)}{\sqrt{n}}$                                      & $2$                                                                                 & $\varepsilon_n\leq \frac{2}{\sqrt{n}}$              & \begin{tabular}[c]{@{}l@{}} (\hyperref[hinv1]{${\mathcal{H}}_{inv,2}^\mathbb{P}$}) \\  (\hyperref[hinv1p]{${\mathcal{H}}_{inv,2}^{L^p}$}) \end{tabular}         & Theorem~\ref{the--1}       
\\ \hline
$\frac{\mathrm{LISC}(\sigma)}{\sqrt{n}}$, $\frac{\mathrm{LDSC}(\sigma)}{\sqrt{n}}$                                      & $2$                                                                                 & $\varepsilon_n\leq \frac{2}{\sqrt{n}}$              &  (\hyperref[hinv1]{${\mathcal{H}}_{inv,2}^\mathbb{P}$})       &   Corollary~\ref{LIcircularcor}    

\\ \hline
\begin{tabular}[c]{@{}l@{}}$\frac{\mathrm{LIS}(\sigma)-2\sqrt{n}}{n^{\frac 16}}$,
\\ $\frac{\mathrm{LDS}(\sigma)-2\sqrt{n}}{n^{\frac 16}}$ \end{tabular}
                          & Tracy-Widom                                                                         & $\varepsilon_n\leq \frac{2}{n^{\frac 16}}$           & (\hyperref[hinv1]{${\mathcal{H}}_{inv,6}^\mathbb{P}$})    & Corollary~\ref{corLIS}   
                          \\ \hline

$\frac{\lambda_i(\sigma)}{\sqrt{n}}$                                & $2$                                                                                 & $\varepsilon_n\leq \frac{4}{\sqrt{n}}$              &\begin{tabular}[c]{@{}l@{}} (\hyperref[hinv1]{${\mathcal{H}}_{inv,2}^\mathbb{P}$}) \\  (\hyperref[hinv1p]{${\mathcal{H}}_{inv,2}^{L^p}$}) \end{tabular}    &    Proposition~\ref{pnRSKEDGE2}    
\\ \hline
$\left(\frac{\lambda_i(\sigma)-2\sqrt{n}}{n^{\frac 16}}\right)_{1\leq i\leq d}$ &Airy ensemble & $\varepsilon_n\leq \frac{4}{n^{\frac 16}}$           &(\hyperref[hinv1]{${\mathcal{H}}_{inv,6}^\mathbb{P}$})     &  Theorem~\ref{Airyens} 
\\ \hline
$s\to \frac{L_{\lambda(\sigma)}(s\sqrt{2n})}{\sqrt{2n}}$            & $\Omega$                                                                            & $\varepsilon'_{n,k} \leq \frac{2\sqrt{k-1}}{\sqrt{n}}$ & (\hyperref[hinv1]{${\mathcal{H}}_{inv,1}^\mathbb{P}$})      & Theorem~\ref{VCthm}  
           \\ \hline
$\frac{\mathcal{K}_j(\sigma)}{n^j} $          &                                       $ \frac{1}{j!^2}                        $                   &    
$\varepsilon_{n} \leq  \frac{2j}{n}$ &       (\hyperref[hinv1]{${\mathcal{H}}_{inv,1}^\mathbb{P}$})  &  Corollary~\ref{big_univ_corr}                     \\ \hline
$\frac{\mathcal{K}_{j}(\sigma) - \frac{n^j}{(j!)^2}}{
\sqrt{n}}$                                                         & $\mathcal{N}\left(0,\frac{{\binom{4j-2}{2j-1}}-2{\binom{2j-1}{j}}^2}{2((2m-1)!)^2}\right)$                                                                                     &     $\varepsilon_{n} \leq  \frac{2j}{\sqrt{n}} $                                                &   (\hyperref[hinv1]{${\mathcal{H}}_{inv,2}^\mathbb{P}$})      &  Corollary~\ref{big_univ_corr}                   
            \\ 
\hline
$\frac{\mathcal{N}_{exc}(\sigma)}{n}$                                  & $\frac{1}{2}$                                                                       &       
$\varepsilon_n\leq \frac{4}{n}$ & (\hyperref[hinv1]{${\mathcal{H}}_{inv,1}^\mathbb{P}$})       & Corollary~\ref{big_univ_corr}      
\\ \hline
$\frac{\mathcal{N}_{exc}(\sigma) -\frac n 2}{
\sqrt{n}}$           & $\mathcal{N}(0,\frac {1}{12})$                                                                            &      
$\varepsilon_n\leq \frac{4}{\sqrt{n}}$                                                 &(\hyperref[hinv1]{${\mathcal{H}}_{inv,2}^\mathbb{P}$})      &  Corollary~\ref{big_univ_corr} 
\\ 

\hline 

$\mathbbm{1}_{D(\sigma)\subset A }$                                                         & $Ber(\det([k_0(j-i)]_{A}))  $&     Proposition~\ref{univ_loc_stat_hat}                                           &   (\hyperref[hinv1]{${\mathcal{H}}_{inv,1}^\mathbb{P}$})      &  Corollary~\ref{det_des_univ}                   \\ 
\hline

$ \frac{\mathcal{N}_{(\tau,X)}(\sigma)-\frac{n^{p-q}}{p!(p-q)!} }{n^{p-q+\frac{1}{2}}} $           & $\mathcal{N}(0,V_{\tau,X})$                                                                            &      
$\varepsilon_n\leq \frac{C}{\sqrt{n}}$                                                 &(\hyperref[hinv1]{${\mathcal{H}}_{inv,2}^\mathbb{P}$})      &  Proposition~\ref{prop:VP} 
\\ 

\hline 

\multicolumn{5}{l}{\textbf{ The results below are fully understood in the  conjugation invariant case.}} 
            \\ 
\hline
$\frac{\mathcal{N}_{D}(\sigma)}{n}$   
& $\frac{1}{2}$                                                                       &       
$\varepsilon_n\leq \frac{4}{n}$ & (\hyperref[hinv1]{${\mathcal{H}}_{inv,1}^\mathbb{P}$})       & Corollary~\ref{big_univ_corr}      
\\ \hline
$\frac{\mathcal{N}_{D}(\sigma) -\frac n 2}{
\sqrt{n}}$   
& $\mathcal{N}(0,\frac {1}{12})$                                                                            &      
$\varepsilon_n\leq \frac{4}{\sqrt{n}}$                                                 &(\hyperref[hinv1]{${\mathcal{H}}_{inv,2}^\mathbb{P}$})      &  Corollary~\ref{big_univ_corr} 
\\ \hline

$\frac{\mathcal{N}_{peak(\sigma)}}{n}$               
&  $\frac{1}{3}$                       
&                   $\varepsilon_n\leq \frac{6}{{n}}$                                   &    (\hyperref[hinv1]{${\mathcal{H}}_{inv,1}^\mathbb{P}$})        &        Corollary~\ref{big_univ_corr}  
\\ \hline
$\frac{\mathcal{N}_{peak}(\sigma) -\frac n 2}{
\sqrt{n}}$       
&      $\mathcal{N}(0,\frac {2}{45})$                                                                               &    $\varepsilon_n\leq \frac{6}{\sqrt{n}}$                                                 &     (\hyperref[hinv1]{${\mathcal{H}}_{inv,2}^\mathbb{P}$})       &     Corollary~\ref{big_univ_corr}      
\\ \hline

$\frac{\mathrm{LAS}(\sigma)}{n}$                                      &                 $\frac{2}{3}$                                                                  &    
$ \varepsilon_n\leq \frac{6}{n}$               &   (\hyperref[hinv1]{${\mathcal{H}}_{inv,1}^\mathbb{P}$})           &    Corollary~\ref{univLASR}  
\\ \hline
$\frac{\mathrm{LAS}(\sigma)-\frac{2n}{3}}{\sqrt{n}}$                                  &            $\mathcal{N}(0,\frac {8}{45})$                                                                          &                                          $\varepsilon_n\leq \frac{6}{\sqrt{n}}$            &       (\hyperref[hinv1]{${\mathcal{H}}_{inv,2}^\mathbb{P}$})     & Corollary~\ref{univLASR}   
\\ \hline

\end{tabular}

\caption{Some examples}
\label{tab:my-table}
\end{table}


\subsection{Proof of theorems~\ref{main_rest1} and \ref{universality_RD}}
\paragraph{}
Let $\rho_n$ be a  conjugation invariant random permutation. To prove theorems \ref{main_rest1} and \ref{universality_RD}, the idea is to modify $\rho_n$ to obtain a conjugation invariant  random permutation supported on $\s^0$. We define  the following Markov operator $T$    associated  to  the uniform random walk over ${\mathcal{G}_{\mathfrak{S}_n}}$. Another way to see it is the following:\footnote{Slightly different Markov operators have already been studied in \citep{kammoun2018,kam2}, we modify a little  the two operators presented in the cited papers to obtain a uniform random walk easy to generalize to other sets. The three operators coincide when $n\leq 3$.  }
\begin{itemize}
 \item If  the realization  $\sigma$ of $\rho_n$ has one cycle, $\sigma$ remains unchanged ($T(\sigma)=\sigma$).
 \item Otherwise, we  choose a couple  $(i,j)$ uniformly from  the nonempty set  $$\{(i,j): j\notin \mathcal{C}_i(\sigma)\}$$
 and we take $T(\sigma)=\sigma \circ (i,j)$. Here $\mathcal{C}_i(\sigma)$ is the cycle of $\sigma$ containing $i$. 
\end{itemize}
For example,  for $n=3$,   transition probabilities of $T$ are  given in Figure \ref{figmM}.   \begin{figure} 
\centering
\begin{tikzpicture}
[-latex,auto ,
semithick,
state/.style ={ circle,top color =white, bottom color = processblue!20,
draw,processblue, text= black, minimum width =1.5cm}]
    \node[state] (s1)  {Id};
    \node[state, below=1cm of s1] (t1) {$(1,2)$};
    \node[state, right=3cm of t1] (t2) {$(2,3)$};
    \node[state, left=3cm of t1] (t3) {$(1,3)$};
    \node[state, below=1cm  of t3] (c1) {$(1,2,3)$};
    \node[state, below= 1cm of t2] (c2) {$(1,3,2)$};
        \draw[every loop,
        line width=0.3mm,
        auto=left,
        >=latex,
        ]
                 (s1) edge[]  node {$\frac{1}{3}$} (t1)
                 (s1) edge[]  node {$\frac{1}{3}$} (t3)
                 (s1) edge[]  node {$\frac{1}{3}$} (t2)
                 (t1) edge[]  node {$\frac{1}{2}$} (c1)
                 (t2) edge[]  node {$\frac{1}{2}$} (c1)
                 (t3) edge[]  node {$\frac{1}{2}$} (c1)
                 (t1) edge[]  node {$\frac{1}{2}$} (c2)
                 (t2) edge[]  node {$\frac{1}{2}$} (c2)
                 (t3) edge[]  node {$\frac{1}{2}$} (c2)
                 (c2) edge[loop below]  node {1} (c2)
                 (c1) edge[loop below]  node {1} (c1);
    \end{tikzpicture}
    \caption{The transition probabilities of $T$ for $n=3$}
    \label{figmM}
\end{figure}
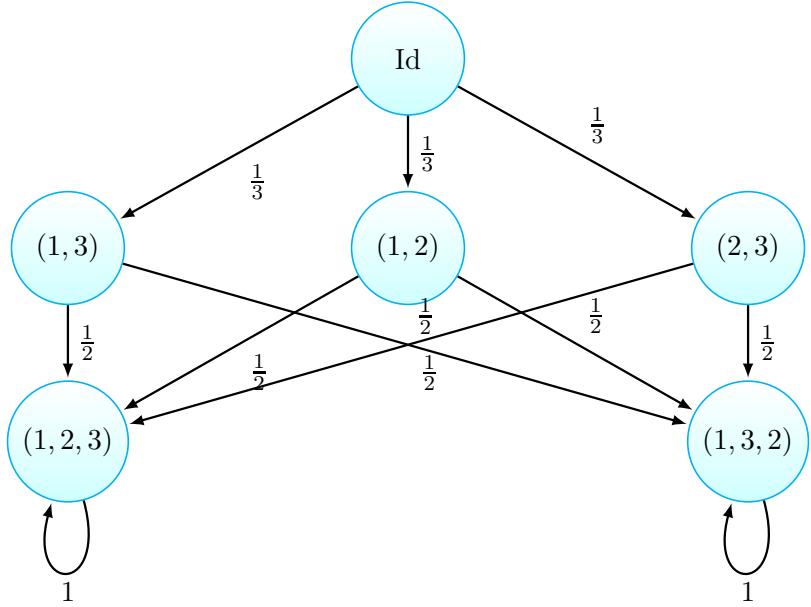
\paragraph{}
We denote by $T^k(\rho_n)$ the random permutation obtained after  applying $k$ times the operator $T$. It is the  random permutation obtained after $k$ steps of the uniform random walk on ${\mathcal{G}_{\mathfrak{S}_n}}$ with initial state $\rho_n$.  Table~\ref{TT1} sums up  the evolution of the random walk  if we start from the uniform distribution on $\mathfrak{S}_3$. 
  Remark that the condition $  j\notin \mathcal{C}_i(\sigma)$ guarantees that $\#(\sigma\circ(i,j))=\#(\sigma)-1$ since the cycles containing $i$ and $j$ are merged and the remaining of cycles are the same for $\sigma $ and $\sigma\circ(i,j).$

In particular, \begin{equation}\label{css1}
\#(T^{i}(\rho_n))\overset{a.s}{=}\max(\#(\rho_n)-i,1).
\end{equation}
The invariant measure of this walk (for conjugation invariant permutations) is trivial.
\begin{lemma}
\label{LEM11}
If $\rho_n$ is a  conjugation invariant random permutation of $\s$ then the law of $T^{n-1}(\rho_n)$ \footnote{After all, a drunk and lost man  who is driving on a two-way road (the Cayley graph of $\mathfrak{S}_n$) needs $n\log(n)$ steps to be close to his destination and will never attend it but if he drives in a one-way road, he needs at most $n$ step to be sure to arrive to destination. In both cases, it is  dangerous for a drunk man to drive. 
 } is the uniform distribution on $\s^0$ i.e.
\begin{align*}
T^{n-1}(\rho_n)\overset{d}{=}\sigma_{Ew,0,n}.
\end{align*}
\end{lemma}

\begin{proof}
\begin{table}
\centering
\begin{tabular}{|l|l|l|l|}
\hline
          & $\sigma_{unif,3}$      & $T(\sigma_{unif,3})$      & $T^2(\sigma_{unif,3})$   \\ \hline
Id        & $1/6$ & $0$            & $0$           \\ \hline
$(1,2)$   & $1/6$ & $1/18$ & $0$           \\ \hline
$(1,3)$   & $1/6$ & $1/18$ & $0$           \\ \hline
$(2,3)$   & $1/6$ & $1/18$ & $0$           \\ \hline
$(1,2,3)$ & $1/6$ & $5/12$ & $1/2$ \\ \hline
$(1,3,2)$ & $1/6$ & $5/12$ & $1/2$ \\ \hline
\end{tabular}
\caption{Transitions for the $\sigma_{unif,3}$ }
\label{TT1}
\end{table}
 First, by construction, if $\rho_n$ is conjugation invariant then $T(\rho_n)$ is also conjugation invariant. Indeed,  one  can see that $T(\rho_n)$ is conjugation invariant since the construction depends only on the cycle structure of $\rho_n$ and all the  integers between $1$ and $n$ play a symmetric role. 
 By iteration,    $T^{n-1}(\rho_n)$ is conjugation invariant. 
Moreover, using  \eqref{css1},   
\begin{align} \label{css2}
\#(T^{n-1} \left(\sigma_{n})\right)\overset{a.s}{=}1. \end{align}  
Knowing that all the elements of $\mathfrak{S}_n^0$  belong to the same conjugacy class, they are equally distributed and  Lemme~\ref{LEM11} follows from  \eqref{css2}.
\end{proof}
We now prove  theorems~\ref{main_rest1} and \ref{universality_RD}.
\begin{proof}[Proof of theorems~\ref{main_rest1} and \ref{universality_RD}]
Equality~\eqref{css1} implies that 
$$T^{n-1}(\rho_n)\overset{a.s}= T^{\#(\rho_n)-1}(\rho_n).$$ Therefore, almost surely,
\begin{align*}
d_F(f(T^{n-1}(\rho_n)),f(\rho_n))=
d_F(f(T^{\#(\rho_n)-1}(\rho_n)),f(\rho_n))
\leq      \varepsilon'_{n,\#(\rho_n)}.
\end{align*}
Thus, if   $\varepsilon'_{n,\#(\rho_n)}\xrightarrow[n\to\infty]{\p}0$, then  for any $\varepsilon>0$, 
\begin{align}\label{end1}  \mathbb{P}\left(d_F\left(f(T^{n-1}(\rho_n)),f(\rho_n)\right)> \varepsilon\right) \xrightarrow[n\to\infty]{} 0.\end{align}  
{ According to  Lemma \ref{LEM11}, $T^{n-1}(\rho_n)$ does not depend on the law of $\rho_n$.
By choosing at first $\rho_n=\sigma_{ref,n}$, \eqref{control_erreur}  then yields  
$$f(\sigma_{Ew,0,n}) \xrightarrow[n\to\infty]{\mathbb{P}} x. $$ 
 By  choosing at a second step $\rho_n=\sigma_n$,  we obtain \eqref{conv1}  for any  $\sigma_n$ satisfying the hypothesis of Theorem~\ref{main_rest1}.
One can prove  Theorem~\ref{universality_RD} using the same argument.
}
 \end{proof}
 \section{Proof of Corollary ~\ref{first_bound_cor}}
 \subsection{First application: Longest
 Increasing Subsequence} \label{PR11}\paragraph{}
 Given $\sigma \in \mathfrak{S}_n$, a subsequence $(\sigma(i_1),\dots,\sigma(i_k))$ is an increasing (resp. decreasing) subsequence of $\sigma$ of length $k$ if $i_1<\dots<i_k$ and $\sigma(i_1)<\dots<\sigma(i_k)$ (resp. $\sigma(i_1)>\dots>\sigma(i_k)$). We denote by $\mathrm{LIS}(\sigma)$ (resp. $\mathrm{LDS}(\sigma)$) the length of the longest increasing (resp. decreasing) subsequence of $\sigma$\footnote{There is a language abuse here: a longest increasing subsequence may not be unique but its length is always defined.}. For example, \label{def:lislds} \begin{equation*}
\text{if}\quad \sigma=\begin{pmatrix}
 1& 2 & 3 & 4 & 5 \\ 
 5& 3 & 2 & 1 & 4 
\end{pmatrix}, \, \mathrm{LIS}(\sigma)=2\,\text{ and }\, \mathrm{LDS}(\sigma)=4. \end{equation*} 
The study of the limiting behavior of $\mathrm{LIS}(\sigma_{unif,n})$, \label{def:unifperm} where $\sigma_{unif,n}$ is a uniform random permutation on $\s$, is known as the Ulam's problem (or the Ulam-Hammersley problem): \citet{MR0129165} conjectured that the limit as $n$ goes to infinity of 
\begin{equation*}
 \frac{\mathbb{E}(\mathrm{LIS}(\sigma_{unif,n}))}{\sqrt{n}}
\end{equation*}
exists. Using a subadditivity argument, \citet{hammersley1972} proved this conjecture. He also proved that this convergence holds in probability.
\citet{MR0480398} and \citet{LOGAN1977206} proved that this limit is equal to $2$. An alternative proof is given by \citet{MR1355056}. 
The asymptotic fluctuations were studied by Baik, Deift and Johansson. They proved the following result: 
\begin{theorem} \citep*{Baik} \label{dbj}
For all $s \in \mathbb{R}$,
\begin{equation*} 
 \mathbb{P}\left(\frac{\mathrm{LIS}(\sigma_{unif,n})-2\sqrt{n}}{n^{\frac{1}{6}}}\leq s\right)\xrightarrow[n\to \infty]{}F_2(s),
\end{equation*}
where $F_2$ \label{def:F2tracy} is the cumulative distribution function (CDF) of the GUE Tracy-Widom distribution. 
\end{theorem}
 For historical details, full proofs and applications, we strongly recommend \citep{MR3468738}. 
Apart from the uniform case, \citet{Mueller2013} studied the longest increasing subsequence for Mallows' distribution. The case of random involutions is studied by \citet{MR1845180} who showed that the limiting distribution depends on the number of fixed points and in some regimes, the GOE/GSE Tracy-Widom distributions appear. They also showed the appearance of a family of probability distributions that interpolate between the GOE and the GSE Tracy-Widom distribution. 
 Mueller and Starr showed that for Mallows' distribution, there is a phase transition between the Gaussian and the Tracy-Widom regimes. In this section, we  prove universality results for conjugation invariant random permutations. 
\begin{theorem} \label{the--1}
Under \normalfont{(\hyperref[hinv1]{${\mathcal{H}}_{inv,2}^\mathbb{P}$})},
$$ \frac{\mathrm{LIS}(\sigma_n)}{\sqrt n} \xrightarrow[n\to\infty]{\mathbb{P}}2 \quad \text{and} \quad \frac{\mathrm{LDS}(\sigma_n)}{\sqrt n} \xrightarrow[n\to\infty]{\mathbb{P}}2.$$ 
Moreover, for any $p\in[1,\infty)$, under \normalfont{(\hyperref[hinv1]{${\mathcal{H}}_{inv,2}^{\mathbb{L}^p}$})},
$$ \frac{\mathrm{LIS}(\sigma_n)}{\sqrt n} \xrightarrow[n\to\infty]{\mathbb{L}^p}2 \quad \text{and} \quad \frac{\mathrm{LDS}(\sigma_n)}{\sqrt n} \xrightarrow[n\to\infty]{\mathbb{L}^p}2.$$ 
\end{theorem}
The convergence in probability is stated without proof in \citep{kammoun2018} as it  is similar to the proof of \citep[Theorem~1.2]{kammoun2018}. 
For the fluctuations, we have the following result.
\begin{theorem} \label{the1}
Assume that $(\sigma_n)_{n\geq 1}$ is conjugation invariant and 
\begin{align}
    \label{conditionbiz}
 \frac{1}{n^{\frac{1}{6}}}
\min_{1\leq i\leq n} \left(\left(\sum_{j=1}^i \#_j(\sigma_n)\right) + \frac{\sqrt{n}}{i}\sum_{j=i+1}^n \#_j(\sigma_n)\right) \xrightarrow[n\to \infty]{\mathbb{P}}0.
\end{align}
\label{def:cyj}
Then for all $s \in \mathbb{R}$,
\begin{equation} \label{TW} \tag{TW}
 \mathbb{P}\left(\frac{\mathrm{LIS}(\sigma_n)-2\sqrt{n}}{n^{\frac{1}{6}}}\leq s\right) \xrightarrow[n\to \infty]{} F_2(s)\, \text{and}\, \mathbb{P}\left(\frac{\mathrm{LDS}(\sigma_n)-2\sqrt{n}}{n^{\frac{1}{6}}}\leq s\right)\xrightarrow[n\to \infty]{} F_2(s).
\end{equation}
Here, $\#_j(\sigma)$ is the number of cycles of $\sigma$ of length $j$. 
\end{theorem}
The idea of the proof we give in Subsection~\ref{proof1} is to construct a coupling between any distribution satisfying these hypothesises and the uniform distribution in order to use Theorem~\ref{dbj} to obtain first the lower bound then the upper bound.
This  theorem generalizes the following result.
\begin{corollary}\citep[Theorem 1.2]{kammoun2018} \label{corLIS}
If \normalfont{(\hyperref[hinv1]{${\mathcal{H}}_{inv,6}^\mathbb{P}$})}
is satisfied 
 then \eqref{TW} holds.
\end{corollary}
 The key argument of our  proofs is the following lemma:
\begin{lemma} \label{lem_controle_lis}
For any permutation $\sigma$ and for any transposition $\tau$, 
\begin{equation*}
|\mathrm{LIS}(\sigma \circ \tau )-\mathrm{LIS}(\sigma)|\leq 2,
\quad |\mathrm{LDS}(\sigma \circ \tau )-\mathrm{LDS}(\sigma)|\leq 2.
\end{equation*}

\end{lemma}
\begin{proof}
Let $\sigma$ be  a permutation.  By definition of $\mathrm{LIS}(\sigma)$,  there exists ${i_1<i_2<\dots<i_ {\mathrm{LIS}(\sigma)}}$ such that ${\sigma(i_1)<\dots<\sigma(i_{\mathrm{LIS}(\sigma)})}$. Let $\tau=(j,k)$  be a transposition
and $i'_1,i'_2,\dots,i'_m$ be   the same sequence as $i_1,i_2,\dots, i_{\mathrm{LIS}(\sigma)}$ after removing $j$ and $k$ if needed. We have  $\sigma(i'_1)<\dots<\sigma(i'_{m})$. In particular,  $\mathrm{LIS}(\sigma)-2\leq m\leq \mathrm{LIS}(\sigma)$. Knowing that  $\forall i\notin \{j,k\}$, $\sigma\circ\tau (i)=\sigma(i)$, then  $$\sigma\circ \tau (i'_1)<\dots<\sigma\circ \tau (i'_{m}).$$ Therefore,
\begin{align}\label{firstone}
\mathrm{LIS}(\sigma)-\mathrm{LIS}(\sigma \circ \tau )\leq 2.
\end{align}
We obtain the second inequality by replacing $\sigma$ by $\sigma \circ \tau$ in \eqref{firstone}.
 For $\mathrm{LDS}(\sigma)$ the proof is similar. 
\end{proof}
\begin{proof}[Proof  of Theorem~\ref{the--1} and  Corollary~\ref{corLIS}  ]
The main functions we want to study are  $$f_{\mathrm{LIS1}}(\sigma):=\frac{\mathrm{LIS}(\sigma)}{\sqrt{n}} \,\text{ and }\, f_{\mathrm{LIS2}}(\sigma):=\frac{\mathrm{LIS}(\sigma)-2\sqrt{n}}{n^\frac{1}{6}}.$$
Using Lemma~\ref{lem_controle_lis}, we have for all $ n\geq 3$, $$\varepsilon_n(f_{\mathrm{LIS1}})=\frac{2}{\sqrt{n}}\, \text{ and }\, \varepsilon_n(f_{\mathrm{LIS2}})=\frac{2}{n^\frac{1}{6}}, $$
and one can conclude using theorems~\ref{main_rest1} and \ref{universality_RD}  with $\sigma_{ref,n}=\sigma_{unif,n}$ since the uniform case is already studied. Indeed, one can see \citep{MR0480398,LOGAN1977206} for the convergence of $f_{\mathrm{LIS1}}$ in probability, \citep*{baik2016combinatorics} for the convergence in $\mathbb{L}^p$ of $f_{\mathrm{LIS1}}$ and    \citep{Baik} for the convergence of $f_{\mathrm{LIS2}}$ in probability.  For the  $\mathrm{LDS}(\sigma)$, the proof is similar.  
\end{proof}
A similar application is the length of the longest increasing (resp. decreasing)  circular subsequence. 
 \begin{definition}
{ Given $\sigma\in\mathfrak{S}_{\infty}$, a subsequence is said to be  increasing (resp. decreasing) circular  if it is increasing (resp. decreasing) up to a circular permutation. One can see \citep{MR2276168} for rigorous
 definition and more details.}
\label{def:LDSC}

 We denote by $\mathrm{LICS}(\sigma)$ (resp. $\mathrm{LDCS}(\sigma)$)  the length of the longest increasing (resp. decreasing)  circular subsequence.  
\end{definition}
%
\begin{corollary}\label{LIcircularcor}
If    \normalfont{(\hyperref[hinv1]{${\mathcal{H}}^\p_{inv,2}$})}
is satisfied 
 then
$$ \frac{\mathrm{LICS}(\sigma_n)}{\sqrt n} \xrightarrow[n\to\infty]{\mathbb{P}}2 \quad \text{and} \quad \frac{\mathrm{LDCS}(\sigma_n)}{\sqrt n} \xrightarrow[n\to\infty]{\mathbb{P}}2.$$ 
\end{corollary}
\begin{proof}
The uniform case is proved in \citep[Theorem 1]{MR2276168} for the $\mathrm{LICS}$ and the case of the $\mathrm{LDCS}$ can be obtained by composition by the permutation $i\mapsto n-i+1$. Moreover, using the same argument as for the  $\mathrm{LIS}$ in Lemma~\ref{lem_controle_lis}, we have
\begin{equation*}
|\mathrm{LICS}(\sigma \circ \tau )-\mathrm{LICS}(\sigma)|\leq 2 
\quad \text{ and } \quad   |\mathrm{LDCS}(\sigma\circ \tau )-\mathrm{LDCS}(\sigma )|\leq 2,
\end{equation*}
which concludes the proof using Theorem~\ref{main_rest1}. 
\end{proof}
We will give now a generalization for the universality for the  $\mathrm{LCS}$.
Given $\sigma\in\s$, let $(\lambda_i(\sigma))_{i\geq 1}$ and $(\lambda'_i(\sigma))_{i\geq 1}$ be respectively the shape of  image of $\sigma$ by the RSK correspondence and its transpose. One way to define it is the following.  Let  \begin{align*}
\mathfrak{I}_1(\sigma):&=\{s\subset\{1,2,\dots,n\};\; \forall i,j \in s,\; (i-j)(\sigma(i)-\sigma(j))\geq 0 \},
\\ \mathfrak{D}_1(\sigma):&=\{s\subset\{1,2,\dots,n\};\; \forall i,j \in s,\; (i-j)(\sigma(i)-\sigma(j))\leq 0 \},
\\\mathfrak{I}_{k+1}(\sigma):&=\{s\cup s',\; s\in \mathfrak{I}_k,\;s'\in \mathfrak{I}_1\},
\\ \mathfrak{D}_{k+1}(\sigma):&=\{s\cup s',\; s\in \mathfrak{D}_k,\;s'\in \mathfrak{D}_1\}.
\end{align*}
For example, for 
$$\sigma_{ex,3} := \begin{pmatrix} 
1 & 2 & 3 \\
2 & 3 & 1 
 \end{pmatrix}, \quad \mathfrak{I}_1(\sigma_{ex,3})= \{\emptyset, \{1\},\{2\},\{3\}, \{1,2\} \} $$
 and 
 $$\mathfrak{I}_2(\sigma_{ex,3})= \mathfrak{D}_2(\sigma_{ex,3})=\{\emptyset, \{1\},\{2\},\{3\}, \{1,2\},\{1,3\},\{2,3\},\{1,2,3\} \}. $$
 \\
The RSK image is defined as follows. 
For any permutation $ \sigma\in \mathfrak{S}_n$,
\begin{align}\label{RSKLEMMAeq}
\max_{s\in \mathfrak{I}_i(\sigma)} \mathrm{card}(s) =\sum_{k=1}^i \lambda_k(\sigma), \quad
\max_{s\in \mathfrak{D}_i(\sigma)} \mathrm{card}(s) =\sum_{k=1}^i \lambda'_k(\sigma).
\end{align}
In particular, $$\max_{s\in \mathfrak{I}_1(\sigma)} \mathrm{card}(s) =\lambda_1(\sigma)=\mathrm{LIS}(\sigma),\quad \max_{s\in \mathfrak{D}_1(\sigma)} \mathrm{card}(s) =\lambda'_1(\sigma)=\mathrm{LDS}(\sigma).$$
We strongly recommend \citep{Sagan2001}  equivalent constructions.\\ 
A more general version of the result of Theorem \ref{dbj} is the following. 
\begin{theorem}\citep[Theorem 5]{Borodin2000}\citep[Theorem 1.4]{MR1826414} \label{BOOJ}
For all real numbers $s_1,s_2,\dots,s_k$,
\begin{align*}
\lim_{n\to \infty}\mathbb{P}\left(\forall i\leq k, \;\frac{\lambda_i(\sigma_{unif,n})-2\sqrt{n}}{{n}^{\frac1 6}}\leq s_i\right)
= F_{2,k}(s_1,s_2,\dots,s_k).
\end{align*}
\end{theorem}
 For the permutations satisfying the  same assumptions  as in  Theorem \ref{the1}, we have the same asymptotic as in the uniform setting at the edge.
 \begin{theorem}\label{Airyens}
Assume that  $(\sigma_n)_{n\geq 1}$ is conjugation invariant and 
\begin{align} \label{condition_bizarre}
 \frac{1}{{n}^{\frac1 6}}
\min_{1\leq i\leq n} \left(\left(\sum_{j=1}^i \#_j(\sigma_n)\right) + \frac{\sqrt{n}}{i}\sum_{j=i+1}^n \#_j(\sigma_n)\right) \xrightarrow[n\to \infty]{\mathbb{P}}0.\end{align}

 Then for all positive integer $k$, for all real numbers $s_1,s_2,\dots,s_k$,
\begin{align}\nonumber
\lim_{n\to \infty}\mathbb{P}\left(\forall i\leq k,\frac{\lambda_i(\sigma_n)-2\sqrt{n}}{{n}^{\frac1 6}}\leq s_i\right)&=\lim_{n\to \infty}\mathbb{P}\left(\forall i\leq k,\frac{\lambda'_i(\sigma_n)-2\sqrt{n}}{{n}^{\frac1 6}}\leq s_i\right)
\\&= F_{2,k}(s_1,s_2,\dots,s_k) \tag{Ai} \label{TW2} .\end{align}
\end{theorem}
Before proving this result,  we recall  first an already known weaker version. 
\begin{proposition} \citep{kammoun2018} \label{propositionsimple}
 \label{pnRSKEDGE}
If    \normalfont{(\hyperref[hinv1]{${\mathcal{H}}_{inv,6}^\mathbb{P}$})}
is satisfied 
 then \eqref{TW2} holds true.
\end{proposition}
Under weaker assumptions, one can still prove the first order convergence. 
\begin{proposition}   \label{pnRSKEDGE2}
If    \normalfont{(\hyperref[hinv1]{${\mathcal{H}}_{inv,2}^\mathbb{P}$})}
is satisfied 
 then for any $i\geq 1$
$$ \frac{\lambda_i(\sigma_n)}{\sqrt n} \xrightarrow[n\to\infty]{\mathbb{P}}2 \quad \text{and} \quad \frac{\lambda'_i(\sigma_n)}{\sqrt n} \xrightarrow[n\to\infty]{\mathbb{P}}2.$$ 
Moreover, for any $p\in[1,\infty)$, under \normalfont{(\hyperref[hinv1]{${\mathcal{H}}_{inv,2}^{\mathbb{L}^p}$})},
$$ \frac{\lambda_i(\sigma_n)}{\sqrt n} \xrightarrow[n\to\infty]{\mathbb{L}^p}2 \quad \text{and} \quad \frac{\lambda'_i(\sigma_n)}{\sqrt n} \xrightarrow[n\to\infty]{\mathbb{L}^p}2.$$ 
\end{proposition} 
Corollary~\ref{corLIS} (resp. Theorem~\ref{the--1}) 
 is a direct application of Proposition~\ref{pnRSKEDGE} (resp. Proposition~\ref{pnRSKEDGE2}) for $k=1$ (resp. $i=1$).
 We will prove first in the next subsection  propositions~\ref{pnRSKEDGE} and~\ref{pnRSKEDGE2} as they are  direct applications of Theorem~\ref{main_rest1}. 
\\ 
 The typical shape  of  $(\lambda_i(\sigma_{unif,n}))_{i\geq 1}$ seen as young diagram  was studied separately by  \citet{LOGAN1977206} and  \citet{MR0480398}. Stronger results are proved by  \citet{Vershik1985}.  In 1993, Kerov studied the limiting fluctuations  but did not publish his results. See \citep{10.1007/978-94-010-0524-1_3} for further  details.
 Let $L_{\lambda(\sigma)}$ be the height function of $\lambda(\sigma)=(\lambda_i(\sigma))_{i\geq 1}$   rotated by $\frac{5\pi}{4}$  and extended by the function $x\mapsto |x|$ to obtain a function defined on $\mathbb{R}$. For example,  if $\lambda(\sigma)=(7,5,2,1,1,\underline{0})$ the associated function $L_{\lambda(\sigma)}$ is represented by Figure \ref{figL}. A direct application of Theorem~\ref{main_rest1} is the following.
\begin{figure}[ht]
\centering
\begin{tikzpicture}    [/pgfplots/y=0.5cm, /pgfplots/x=0.5cm]
      \begin{axis}[
    axis x line=center,
    axis y line=center,
    xmin=0, xmax=10,
    ymin=0, ymax=10, clip=false,
    ytick={0},
    xtick={0},
    minor xtick={0,1,2,3,3,4,5,6,7,8,9},
    minor ytick={0,1,2,3,3,4,5,6,7,8,9},
    grid=both,
    legend pos=north west,
    ymajorgrids=false,
    xmajorgrids=false, anchor=origin,
    grid style=dashed    , rotate around={45:(rel axis cs:0,0)}
,
]

\addplot[
    color=blue,
        line width=3pt,
    ]
    coordinates {
    (0,10)(0,7)(1,7)(1,5)(2,5)(2,2)(3,2)(3,1)(5,1)(5,0)(10,0)
    };
 
\end{axis}
\begin{axis}[
    axis x line=center,
    axis y line=center,
    xmin=-7.07, xmax=7.07,
    ymin=0, ymax=8, anchor=origin, clip=false,
    xtick={-7,-6,-5,-4,-3,-2,-1,0,1,2,3,4,5,6,7},
    ytick={0,1,2,3,3,4,5,6,7,8},
    legend pos=north west,
    ymajorgrids=false,
    xmajorgrids=false,rotate around={0:(rel axis cs:0,0)},
    grid style=dashed];
\end{axis}
    \end{tikzpicture}
    \caption{ $L_{(7,5,2,1,1,\underline{0})}$}
     \label{figL}
\end{figure}
 \begin{theorem}\citep{kammoun2018} \label{VCthm} 
Under   \normalfont{(\hyperref[hinv1]{${\mathcal{H}}^\p_{inv,1}$})},
\begin{align}\tag{VKLS}\label{VC}
 \sup_{s\in \mathbb{R}} \left|\frac{1}{\sqrt{2n}}L_{\lambda(\sigma_{n})}\left({s}{\sqrt{2n}}\right)-\Omega(s)\right| \xrightarrow[n\to \infty]{\mathbb{P}}0,
\end{align}
where,
\begin{align*}
\Omega(s):=\begin{cases}
\frac{2}{\pi}(s\arcsin({s})+\sqrt{1-s^2}) & \text{ if } |s|<1 \\ 
|s| & \text{ if } |s|\geq 1 
\end{cases}.
\end{align*}

\end{theorem}
\begin{proof}
We wan to apply Theorem~\ref{main_rest1}. Let now $F$ be the set of continual diagrams i.e. the set of  $1$-Lipschitz real functions $g$ from $\mathbb{R}$ to $\mathbb{R}_+$ such that $\exists  a,b>0$ s.t. $\forall x\notin [-b,b], g(x)=|x-a|$. For $g,h\in F$, we denote by  $d_F(g,h)=\sup_\mathbb{R}|h-g|$.  
For $\sigma\in\s$, $f(\sigma)$ is the function $ s\to \frac{L_{\lambda(\sigma)}\left({s}{\sqrt{2n}}\right)}{\sqrt{2n}}$. So that $f$ is a function from $\mathfrak{S}_\infty$  taking values in the metric space $(F,d_F)$.
If we choose $\sigma_{ref,n}=\sigma_{unif,n}$ and $x$ to be the function $\Omega$, the convergence 
$$ f(\sigma_{ref,n}) \xrightarrow[n\to\infty]{\p}x$$
is proven  by  \citet{LOGAN1977206} and  \citet{MR0480398}.
Using \cite[Lemma 3.7.]{kammoun2018}, for any $1\leq k\leq n$,
\begin{equation}
\varepsilon'_{n,k}(f)\leq 
2\sqrt{\frac{k-1}{n}}.
\end{equation}
So that Theorem~\ref{main_rest1} gives the conclusion.

\end{proof}
\subsection{Proof of propositions~\ref{pnRSKEDGE} and~\ref{pnRSKEDGE2}}

\begin{lemma} \label{lemma2}
For any permutation $\sigma$ and any transposition  $\tau$,\begin{equation} \label{sum}
\left|\sum_{k=1}^i \lambda_k(\sigma)-{\lambda}_k\left(\sigma\circ\tau\right)\right| \leq 2\quad\text{and}\quad
\left|\sum_{k=1}^i \lambda'_k(\sigma)-\lambda'_k\left(\sigma\circ\tau\right)\right| \leq 2.
\end{equation}
Moreover, 
\begin{equation} \label{sep}
\left|\lambda_i(\sigma)-\lambda_i\left(\sigma\circ\tau\right)\right| \leq 4\quad\text{and}\quad
\left|\lambda'_i(\sigma)-\lambda'_i\left(\sigma\circ\tau\right)\right| \leq 4.
\end{equation}
\end{lemma}
\begin{proof} Let $\sigma$ be a permutation and $\tau=(l,m)$ be a transposition.  We have then for all integer $i$,
\begin{equation*}
\{s\setminus{\{l,m\}},s\in \mathfrak{I}_i(\sigma)\}\subset \mathfrak{I}_i(\sigma\circ\tau)
\end{equation*}
and similarly  
\begin{equation*}
\{s\setminus{\{l,m\}},s\in \mathfrak{D}_i(\sigma)\}\subset \mathfrak{D}_i(\sigma\circ\tau).
\end{equation*}
Consequently, using \eqref{RSKLEMMAeq},
\begin{equation*}
\sum_{k=1}^i \lambda_k(\sigma)-{\lambda}_k(\sigma\circ\tau) \geq -2
, \quad \sum_{k=1}^i \lambda'_k(\sigma)-\lambda'_k(\sigma\circ\tau) \geq -2.
\end{equation*}
Using the same argument with $\sigma \circ \tau $ instead of $\sigma$, \eqref{sum}  follows. Moreover, since $$\lambda_{i+1}=\sum_{k=1}^{i+1}\lambda_k-\sum_{k=1}^i\lambda_k, \quad \lambda'_{i+1}=\sum_{k=1}^{i+1}\lambda'_k-\sum_{k=1}^i\lambda'_k,$$ 
the triangle inequality yields   \eqref{sep}.
\end{proof}
Using   \eqref{sep}, Propositions~\ref{pnRSKEDGE} and~\ref{pnRSKEDGE2} are direct applications of Theorem~\ref{main_rest1}.

\subsection{Second application: Longest Alternating Subsequence}
\paragraph{}
A more tricky application is the length of the Longest Alternating Subsequence. This is  a special case of a  large class of statistics we will present in the next subsection. 
 \begin{definition}
 Given  $\sigma\in\s$,  $(\sigma(i_1),\sigma(i_2),\dots,\sigma_n(i_k))$ is said to be an alternating subsequence of $\sigma$ of length $k$ if $i_1<i_2<\dots<i_k$ and
 $\sigma(i_1)>\sigma(i_2)<\sigma(i_3)>\dots\sigma(i_k)$.
 We denote by $\mathrm{LAS}(\sigma)$  the length of the  longest alternating subsequence of $\sigma$. \label{def:LAS}
 \end{definition}
 The uniform case is already studied in \citep{MR2757798,MR2820763}. We have the two following results. \begin{proposition}\citep[Page 17]{MR2757798} For $n\geq 2$, 
 \begin{align*}
 \E(\mathrm{LAS}(\sigma_{unif,n}))=\frac{2n}{3}+\frac{1}{6}
 \end{align*}
and  for $n\geq 4$,
 \begin{align*}
     \mathbb{V}\mathrm{ar}(\mathrm{LAS}(\sigma_{unif,n}))=\frac{8n}{45}-\frac{13}{180}.
 \end{align*}
 \end{proposition}
 \begin{proposition}\citep[Proposition 4]{MR2820763}
 $$\frac{\mathrm{LAS}(\sigma_{unif,n})-\frac{2}{3}n}{\sqrt{n}}\xrightarrow[n\to\infty]{d}\mathcal{N}\left(0,\frac{8}{45}\right).$$
 \end{proposition}
Here, $\mathcal{N}(m,\sigma^2)$ is the normal distribution. 
We also make  use of the following result. 
\begin{proposition}\citep[Corollary 2]{MR2820763} \label{prop_carac_LAS}
$$\mathrm{LAS}(\sigma) = 1+ \sum_{i=1}^{n-1} M_k(\sigma),$$
where
$$M_1(\sigma)= \mathbbm{1}_{\sigma(1)>\sigma(2)}$$
and for $1<k<n$,
$$ M_k(\sigma)= \mathbbm{1}_{\sigma(k-1)>\sigma(k)<\sigma(k+1)} +\mathbbm{1}_{ \sigma(k-1)<\sigma(k)>\sigma(k+1) }. $$
\end{proposition}
This yields the following.  
\begin{lemma}\label{lemmaLAS}  
For any $\sigma\in\s$ and  $1\leq i,j\leq n$,
$$|\mathrm{LAS}(\sigma)-\mathrm{LAS}(\sigma\circ(i,j))|\leq 6.$$ 
\end{lemma}
\begin{proof} Let  $1\leq k<n$.
If $\min(|k-i|,|k-j|)\geq 2$, then $M_k(\sigma)=M_k(\sigma\circ(i,j))$ and consequently, \label{def:card}
\begin{align*}
|\mathrm{LAS}(\sigma)-\mathrm{LAS}(\sigma\circ(i,j))|& = \left|\sum_{k\in(\{i-1,i,i+1\}\cup\{j-1,j,j+1\})\cap\{1,\dots,n-1\} } M_k(\sigma)-M_k(\sigma\circ(i,j))\right|
\\&\leq \sum_{k\in(\{i-1,i,i+1\}\cup\{j-1,j,j+1\})\cap\{1,\dots,n-1\} }\left| M_k(\sigma)-M_k(\sigma\circ(i,j))\right|
\\&\leq \sum_{k\in(\{i-1,i,i+1\}\cup\{j-1,j,j+1\})\cap\{1,\dots,n-1\} }1
\\&=\mathrm{card}((\{i-1,i,i+1\}\cup\{j-1,j,j+1\})\cap\{1,\dots,n-1\})
\\& \leq 6.
\end{align*}
\end{proof}
Consequently, we have the next corollary. 
\begin{corollary} \label{univLASR}
\begin{itemize}
     \item   Under  {\normalfont{(\hyperref[hinv1]{${\mathcal{H}}^\p_{inv,1}$})}}, we have 
    \begin{align}\label{eqLAS1} \frac{\mathrm{LAS}(\sigma_n)}{n} \xrightarrow[n\to\infty]{\mathbb{P}}\frac{2}{3}\end{align}
and  
\begin{align} \label{eqLAS2} \mathbb{E}(\mathrm{LAS}(\sigma_n))=\frac{2}{3}n+o(n).\end{align}
     \item  Under  {\normalfont{(\hyperref[hinv1]{${\mathcal{H}}^\p_{inv,2}$})}},  we have    \begin{align} \label{eqLAS3}
        \frac{\mathrm{LAS}(\sigma_{n})-\frac{2}{3}n}{\sqrt{n}}\xrightarrow[n\to\infty]{d}\mathcal{N}\left(0,\frac{8}{45}\right).
    \end{align}
\end{itemize}
\end{corollary}
\begin{proof}[Proof of Corollary~\ref{univLASR}]
 Let  $f_{LAS1}$ and $f_{LAS2}$ be the two functions defined on $\mathfrak{S}_{\infty}$ by: \\  For $\sigma\in\s$, 
$$f_{LAS1}(\sigma):=\frac{\mathrm{LAS}(\sigma)}{n}\quad \text{and}\quad f_{LAS2}(\sigma):= \frac{\mathrm{LAS}(\sigma)-\frac{2}{3}n}{\sqrt{n}}.$$
By Lemma~\ref{lemmaLAS}, we obtain $\varepsilon_n(f_{LAS1})\leq \frac{6}{n} $ and  $\varepsilon_n(f_{LAS2})\leq \frac{6}{\sqrt{n}}.$  
Thus \eqref{eqLAS1} and \eqref{eqLAS3} follow from theorems~\ref{main_rest1} and \ref{universality_RD}. Moreover, since $\frac{\mathrm{LAS}(\sigma_n)}{n}\in(0,1]$, \eqref{eqLAS2} is a direct consequence of \eqref{eqLAS1}.

\end{proof}
\subsection{Local statistics}  \label{sub:loc}
{
\begin{definition} \label{def-local}
Given $k\geq 1$, we call a function $f$ defined on $\mathfrak{S}_\infty$  a local function of type $k$,  and we write  $f\in \mathcal{L}oc_k $,  if  there exist a positive integer $m\geq 1$, a Boolean function $g$ defined on  $\mathbb{N}^{(m+1)k}$  such that, for any  $n\geq k+m-1$ and  any $ \sigma\in \s$, 
\begin{align*}
f(\sigma)={\sum_{1\leq i_1<\dots<i_k\leq n} g(i_1,\dots,i_k,\sigma(i_1),\sigma(i_1-1),\dots,\sigma(i_1-m+1),\sigma(i_2),\dots,\sigma(i_k-m+1))}.
\end{align*}
We used  the convention $\sigma(i)=0$ when $i\leq 0$. 
\end{definition} }
Here are some examples of local statistics. 
\begin{itemize}
 \item The number of fixed points: 
\\ By choosing $k=m=1$ and 
$g(x,y)= \mathbbm{1}_{x=y}$, we obtain that $\mathrm{tr}\in \mathcal{L}oc_1$.
 \item  $\#_k \in \mathcal{L}oc_k$  and $\sigma\mapsto\mathrm{tr}(\sigma^k)  \in \mathcal{L}oc_k$. 
 \item  \label{def:exc} The number of $j$-exceedances\footnote{In the literature, $j$-exceedances is sometimes defined by the condition $\sigma_{i}\geq i+j$ and othertimes by $\sigma_{i}= i+j$.  In both cases, the number $j$-exceedances is a local statistic but only  the first case is in interest for our purpose.}:\\
For $j\in\mathbb{N}$ fixed, we define for  $\sigma \in \s$ and, we define 
$$\mathcal{N}_{exc_j}(\sigma):=\mathrm{card}(\{i, \sigma_{i}\geq i+j\}). $$
We choose again $k=m=1$ and $g(x,y)= \mathbbm{1}_{x+j\leq y}$ and 
we obtain  again  $\mathcal{N}_{exc_j}\in \mathcal{L}oc_1. $

 \item {Longest alternating subsequence (LAS)}:\\   $\mathrm{LAS}\in\mathcal{L}oc_1$. This is a direct application of Proposition~\ref{prop_carac_LAS}.
Here, $k=1,m=3$ and \begin{align*}
    g(i,y_1,y_2,y_3)=  \begin{cases} 
   0 & \text{if } i=0
\\    1 & \text{if }i=1
\\     \mathbbm{1}_{y_2>y_1}  & \text{if }i=2 
\\     \mathbbm{1}_{l<k>j}+\mathbbm{1}_{y_3>y_2<y_1} & \text{if } i>2
    \end{cases}. 
\end{align*}
 \item {Number of peaks}:\\  \label{def:pik}
For $\sigma \in \s$, we define 
$$\mathcal{N}_{peak}(\sigma):=\mathrm{card}(\{1<i<n, \sigma(i-1)<\sigma(i)>\sigma(i+1) \}). $$
We choose again $k=1,m=3$ and $g(x,y_1,y_2,y_3)= \mathbbm{1}_{x\geq 3}\mathbbm{1}_{y_1<y_2>y_3}$ and 
we obtain  again  $\mathcal{N}_{peak}\in \mathcal{L}oc_1.$

 \item {Number of $j$-descents}: \label{def:des}
\\
For  $ j\geq 1, \sigma \in \s$, we define 
\begin{align*}
    \mathcal{N}_{D_j}(\sigma):= \mathrm{card}\{ 1\leq i\leq n-1,  \sigma(i+1)+j\leq \sigma(i) \}.
\end{align*}
We choose  $k=1,m=2$ and $g(x,y_1,y_2)= \mathbbm{1}_{x\geq 2}\mathbbm{1}_{y_2\geq y_1+j}$ and 
we obtain  again  $\mathcal{N}_{D_j}\in \mathcal{L}oc_1.$

When $j=1$, the $1-$descents are known as the descents. We also set  
    \begin{align*} \mathcal{N}_{D}(\sigma)&:= \mathrm{card}\{ 1\leq i\leq n-1,  \sigma(i+1)< \sigma(i) \}=\mathcal{N}_{D_1}(\sigma).
    \end{align*}

 \item Number of inversions and  $m-$clicks of the permutation graph: 
 
\begin{definition}
Let $\sigma\in\s$. Let $\mathfrak{G}(\sigma)=(V_{\mathfrak{G}(\sigma)},E_{\mathfrak{G}(\sigma)})$\footnote{Fun fact 1: the application  $\sigma \mapsto \mathfrak{G}(\sigma) $ is injective.} be  the permutation graph  of  $\sigma$ defined by  $$ V_{\mathfrak{G}(\sigma)}=\{1,\dots,n\} \, \text{ and } \, E_{\mathfrak{G}(\sigma)} = \{ (i,j)\in \{1,2,\dots,n\}; (\sigma(i)-\sigma(j))(i-j)<0 \}.$$ \end{definition}
For example, $E_{\mathfrak{G}(\sigma)}=\emptyset$ if and only if $\sigma=Id_n$ and for the permutation $\sigma: i \mapsto n-i+1 $, $\mathfrak{G}(\sigma)$ is the complete graph with $n$ vertices.

Given  $j\geq2$, we denote by $\label{def:clicks}$ $$\mathcal{K}_j(\sigma):=\mathrm{card}(\{(i_1,i_2,\dots,i_j); 1\leq i_1<\dots<i_j\leq n,\sigma(i_1)>\dots>\sigma (i_j)\})
$$
the number of $j$-clicks of $\mathfrak{G}(\sigma)$\footnote{This a special case of the number of occurrences of a pattern in a permutation. In general, the number of occurrences of any pattern is a local statistic.}.
In particular, $\mathcal{K}_2(\sigma)$ is the number of inversions of $\sigma$. One can easily check  that with $\mathcal{K}_j\in\mathcal{L}oc_j.$ 
Here, $$g(x_1,\dots,x_j,y_1,\dots,y_j)= \mathbbm{1}_{y_1>y_2>\dots>y_j}.$$
 \item  Let $d_k(\sigma):=\mathrm{card}(\{i; (i,k)\in E_{\mathfrak{G}(\sigma)}\})$ be  the degree of the vertex $k$ in $\mathfrak{G}(\sigma)$. We have $d_k(\sigma)\in \mathcal{L}oc_2$.
\end{itemize}
{
\begin{proposition} \label{funny_prop}
Given $k\geq 1$,  $f\in\mathcal{L}oc_k$, a random real variable $X$, $k-1<\gamma\leq k$  and ${(a_n)_{n\geq 0}\in\mathbb{R}^\mathbb{N}}$ such that
\begin{align*}
\frac{f(\sigma_{unif,n})-a_n}{n^\gamma} \xrightarrow[n\to\infty]{d} X,
\end{align*}
if  \normalfont{(\hyperref[hinv1]{${\mathcal{H}}^\p_{inv,\frac{1}{\gamma-k+1}}$})}  holds then  
\begin{align*}
\frac{f(\sigma_{n})-a_n}{n^\gamma} \xrightarrow[n\to\infty]{d} X.
\end{align*}
\end{proposition}
\begin{proof}
By counting the number of possible choices of   $1\leq i_1<i_2,\dots<i_k\leq n$ such that  $\{i,j\} \cap \{i_1,\dots,i_{1}-m+1,i_{2},\dots, i_{k}-m+1\} \neq \emptyset $,
it is easy to see that for  any permutation $\sigma\in\s$ and any transposition $(i,j)$ we have
\begin{align*}
    |f(\sigma(i,j))-f(\sigma)|\leq  \frac{2km (n-1)!}{(k-1)!(n-k)!}\leq 2km{n^{k-1}}.
\end{align*}
Consequently for $h=\frac{f-a_n}{n^\gamma}$,
$\varepsilon_n(h)\leq  {2k}{n^{k-\gamma -1}}m$ and one can conclude using Remark~\ref{remarque_RD}.
\end{proof}
One can then easily apply this result combined with the discussion in the previous subsection to our local statistics. 
\begin{corollary} \label{big_univ_corr}
Under  \normalfont{(\hyperref[hinv1]{${\mathcal{H}}^\p_{inv,1}$})}, we have for any $j\geq 2$,
\begin{align*}
\frac{\mathcal{N}_{D_{j}}(\sigma_n)}{n}&\xrightarrow[n\to\infty]{\mathbb{L}^1} \frac{1}{2},
\\     \frac{\mathcal{N}_D (\sigma_n) }{n}&\xrightarrow[n\to\infty]{\mathbb{L}^1} \frac{1}{2}, \\ 
\frac{\mathcal{K}_{j}(\sigma_n)}{n^m}&\xrightarrow[n\to\infty]{\mathbb{L}^1} \frac{1}{(m!)^2},
\\ \frac{\mathcal{N}_{exc_{j}}(\sigma_n)}{n}&\xrightarrow[n\to\infty]{\mathbb{L}^1} \frac{1}{2},
\\ \frac{\mathcal{N}_{peak(\sigma_n)}}{n}&\xrightarrow[n\to\infty]{\mathbb{L}^1} \frac{1}{3}.
\end{align*}
Moreover, under 
 \normalfont{(\hyperref[hinv1]{${\mathcal{H}}^\p_{inv,2}$})}, we have for any $j\geq 2$,
\begin{align*}
\frac{\mathcal{N}_{D_{j}}(\sigma_n) -\frac n 2}{
\sqrt{n}}&\xrightarrow[n\to\infty]{d} \mathcal{N}\left(0,\frac{1}{12}\right),\\    \frac{\mathcal{N}_{D}(\sigma_n) -\frac n 2}{
\sqrt{n}}&\xrightarrow[n\to\infty]{d} \mathcal{N}\left(0,\frac{1}{12}\right),  \\ 
\frac{\mathcal{K}_{j}(\sigma_n) - \frac{n^j}{(j!)^2}}{
n^{j-\frac{1}{2}}}&\xrightarrow[n\to\infty]{d} \mathcal{N}(0,v_j),
\\ \frac{\mathcal{N}_{exc_{j}}(\sigma_n) -\frac n 2}{
\sqrt{n}}&\xrightarrow[n\to\infty]{d} \mathcal{N}\left(0,\frac{1}{12}\right),
\\ \frac{\mathcal{N}_{peak}(\sigma_n) -\frac n 2}{
\sqrt{n}}&\xrightarrow[n\to\infty]{d} \mathcal{N}\left(0,\frac{2}{45}\right),
\end{align*}
where 
\begin{align*}
    v_j=\frac{{\binom{4j-2}{2j-1}}-2{\binom{2j-1}{j}}^2}{2((2m-1)!)^2}.
\end{align*}
\end{corollary}
The uniform case  for  $\mathcal{N}_D$, $\mathcal{N}_{peak},  \mathcal{K}_j$  and $\mathcal{N}_{exc_{1}}$ has already been studied. One can find a proof respectively in  \citep{MR3998822}, \citep{fulman2019central},  \citep{grerk2019study} and \citep{MR3091722}. For the conjugation invariant case, as we explained before,  $\mathcal{N}_D$ and  $\mathcal{N}_{peak}$  are fully understood but, to the best knowledge of the author, it is not the case for  $\mathcal{K}_j$  and $\mathcal{N}_{exc_{1}}$. 
For $\mathcal{N}_{exc_{1}}$,  the special case of the Ewens distribution  is studied in \citep{MR3091722}.  Moreover, the results for  $\mathcal{N}_{D_j}$ and $\mathcal{N}_{exc_j}$  are   direct consequences of respectively $\mathcal{N}_{D}$ and $\mathcal{N}_{exc_1}$ since for any conjugation invariant random permutation $\sigma_n$,
$$0\leq \mathbb{E} (\mathcal{N}_{D}(\sigma_n)-\mathcal{N}_{D_j}(\sigma_n))= \frac{(j-1)(n-j-1) (1-\mathbb{P}(\sigma_n(1)=1))}{n-1}\leq j-1 $$ 
and
$$0\leq \mathbb{E} (\mathcal{N}_{exc_1}(\sigma_n)-\mathcal{N}_{exc_j}(\sigma_n)) \leq j-1.$$
\subsection{Number of occurrences of a vincular permutation pattern}
Vincular Patterns also known as dashed patterns are introduces by \cite{Babson2000GeneralizedPP}. We use the same definition as in \citep{MR3091722}.
\begin{definition}
A vincular pattern of size $p$ is a couple $(\tau,X)$ such that  $\tau\in\mathfrak{S}_p$ and $X\subset [p-1]$. Given $\sigma \in \mathfrak{S}_\infty$, an occurrence of $(\tau,X)$ 
is a list $i_1 <  \dots < i_p$ such that
\begin{itemize}
    \item   $i_{x+1} = i_{x} + 1$ for any $x \in X$.
\item  $(\sigma(i_1),\dots,\sigma(i_p))$ is in the same relative order as $(\tau(i_1),\dots,\tau(i_p))$.
\end{itemize}
We denote by $\mathcal{N}_{(\tau,X)}(\sigma)$ the number of occurrences of $(\tau,X)$ in $\sigma$.
\end{definition}
When $X=\emptyset$, $(\tau,X)$ is said to be a classic pattern. Here is some examples of vincular patterns: 
\begin{itemize}
    \item $\mathcal{N}_{(21,\emptyset)}=\mathcal{N}_{inv}$
    \item $\mathcal{N}_{(21,\{1\})}=\mathcal{N}_{D}$
    \item $\mathcal{N}_{(j\dots21,\emptyset)}=\mathcal{K}_j$
    \item $\mathcal{N}_{(132,\{1,2\})}+\mathcal{N}_{(231,\{1,2\})}=\mathcal{N}_{peak}$.
\end{itemize}
Remark that for any $(\tau,X)$, $\mathcal{N}_{(\tau,X)}\in \mathcal{L}oc_p \cap
\mathcal{L}oc_{p-\mathrm{card}(X)}$.
\\
For the uniform case, 
\cite{Bna2010PermutationPO}, \cite*{janson2011random} and \cite{Hofer2017ACL} proved respectively a CLT for monotone, classic and vincular patterns. 
\cite{MR3091722} gives a generalization for the Ewens distribution. In particular, \cite{Hofer2017ACL} proved that
for any $\tau \in \mathfrak{S}_p$ and any $X\subset [p-1]$,
\begin{align*}
    \frac{\mathcal{N}_{(\tau,X)}(\sigma_{unif,n})-\frac{n^{p-q}}{p!(p-q)!} }{n^{p-q-\frac{1}{2}}} \xrightarrow[n\to\infty]{d}\mathcal{N}(0,V_{\tau,X}).
\end{align*}
Here, $q=\mathrm{card}(X)$ and $V_{\tau,X}>0$. Using Proposition~\ref{funny_prop}, we have immediately the following. 
\begin{proposition} \label{prop:VP}
  Under  {\normalfont{(\hyperref[hinv1]{${\mathcal{H}}^\p_{inv,2}$})}},   for any $\tau \in \mathfrak{S}_p$ and any $X\subset [p-1]$ 
\begin{align*}
    \frac{\mathcal{N}_{(\tau,X)}(\sigma_{n})-\frac{n^{p-q}}{p!(p-q)!} }{n^{p-q-\frac{1}{2}}} \xrightarrow[n\to\infty]{d}\mathcal{N}(0,V_{\tau,X}).
\end{align*}
Here, $q=\mathrm{card}(X)$ and $V_{\tau,X}>0$.

\end{proposition}

\section{Further discussion and improved  bounds} \label{sub:lochat}
\subsection{Universality for \texorpdfstring{$\widetilde{\mathcal{L}oc}$}{text}} 
We  denote by  $\widetilde{\mathcal{L}oc}$ the set of  local functions  $f$ of any type associated with a Boolean function $g$ such that  \begin{align}\mathrm{card}(\{i\in \mathbb{N}^*;  
\max_{I\in \mathbb{N}^{k-1}}\max_{J\in \mathbb{N}^{mk}}  g(I,i,J)=1 \})<\infty.
\end{align}
For this class, it is simple to obtain the convergence of the expectation. It can be seen as a macroscopic universality result. \paragraph{}
Let $A\subset\mathbb{N}^*$ be finite, $n> \max(A)$
and $(\sigma_n)_{n\geq 1}$ satisfying \eqref{hinv}.
Using again the random walk associated to $T$ and seeing that$$ \p(\exists i\in  \{i_1-i_2;i_1 \in A, 0\leq i_2<m-1 \}, (T^{n-1}(\sigma_n))(i) \neq \sigma_n(i))\leq  \frac{2\#(\sigma_n) \mathrm{card} (A)m}{n}, $$  we obtain the following.
\begin{proposition} \label{univ_loc_stat_hat}
Given  $f\in \widetilde{\mathcal{L}oc}$ and assuming that  $(\sigma_n)_{n\geq 1}$ and $(\sigma_{ref,n})_{n\geq 1}$ satisfy  \normalfont{(\hyperref[hinv1]{${\mathcal{H}}^\p_{inv,1}$})} we have 
$$ \mathbb{E}(f(\sigma_n))- \mathbb{E}(f(\sigma_{ref,n})) \xrightarrow[n\to\infty]{} 0.$$
Moreover, if  
$f(\sigma_{ref,n})$ converges in distribution then $f(\sigma_{n})$ does also converge to the same limit.
\end{proposition}
We give now an application: Let $n$ be a positive integer and $\sigma \in \mathfrak{S}_n$, we define  
\begin{align} \label{def_desc}
{D}(\sigma):=\{ i\in \{1,\,\dots,\,n-1\};\;\sigma(i+1)<\sigma(i)\}.
\end{align}
\paragraph*{} When $\sigma$ is random, $D(\sigma)$ is known as a descent process. 
\paragraph{}
Given $A \subset \mathbb{N}^*$  finite, if we introduce  
\begin{align}\label{descent_process}
    D^A(\sigma):=\mathbbm{1}_{A\subset D(\sigma)}, \end{align}
then $D^A\in \mathcal{L}oc_{|A|}\cap  \widetilde{\mathcal{L}oc}.$ Here, 
$$g(x_1,x_2,\dots,x_{|A|},y_1,y'_1,y_2,\dots, y_{|A|},y'_{|A|})=\mathbbm{1}_{A=\{x_i-1, 1\leq i\leq |A|\}}\prod_{i=1}^{|A|}\mathbbm{1}_{y_i<y'_i}.  $$
We further investigate  the descent process. First, the descent process is well understood in the uniform case. 
\begin{theorem} 
\citep[Theorem 5.1]{MR2721041}
\label{borodin2} For any positive integer $n$ and any  ${A \subset \{1,2,\dots,n-1\}}$, 
\begin{equation*}\mathbb{P}(A \subset D(\sigma_{unif,n}))=\det([k_{0}(j-i)]_{i,j \in A}),
\end{equation*}
where, 
\begin{align*}
\sum_{i\in \mathbb{Z}}k_0(i)z^i=\frac{1}{1-e^z}.
\end{align*}
\end{theorem}
We say  that the descent process is determinantal with kernel $K_0(i,j):=k_0(j-i)$. 
\paragraph*{} In the non-uniform setting, the descent process is already studied for the Mallows law with  Kendall tau metric: it is also determinantal  with   different kernels, see \citep[Proposition 5.2]{MR2721041}. We  showed in \citep{kammoun2018}  that for a large class of random permutations, the limiting descent process is determinantal with the same kernel as the uniform setting. We will detail a weaker result than  \citep{kammoun2018}. 
\begin{corollary} \label{det_des_univ}
Under  \normalfont{(\hyperref[hinv1]{${\mathcal{H}}^\p_{inv,1}$})}, for any finite set  $A \subset \mathbb{N}^*$, 
\begin{equation}  \tag{DPP}
\lim_{n\to \infty} \mathbb{P}(A \subset D(\sigma_n))=\det([k_0(j-i)]_{i,j \in A}).
\end{equation}
\end{corollary}
\begin{proof} Just apply Proposition~\ref{univ_loc_stat_hat} for the statistic $D^A$ defined in \eqref{descent_process}.
\end{proof} 
The same argument can be applied for other local statistics but not necessarily in $\widetilde{\mathcal{L}oc}$. For example,  we have similar results for the degree of vertices of the permutation graph.
\begin{proposition}   
Under  \normalfont{(\hyperref[hinv1]{${\mathcal{H}}^\p_{inv,1}$})},
\begin{equation*}
 \frac{d_{k}(\sigma_n)}{n} \xrightarrow[n\to\infty]{\p} \frac{1}{2},
\quad
    \frac{d_{\frac{n}{2}}(\sigma_n)}{n} \xrightarrow[n\to\infty]{\p} \frac{1}{2},
\quad     \frac{d_{{n}}(\sigma_n)}{n} \xrightarrow[n\to\infty]{\p} \frac{1}{2}
.
\end{equation*}
Moreover, under  \normalfont{(\hyperref[hinv1]{${\mathcal{H}}^\p_{inv,2}$})}, 
\begin{equation*}
    \frac{d_{\frac{n}{2}}(\sigma_n)-\frac{n}{2}}{2\sqrt{n}} \xrightarrow[n\to\infty]{d} \mathcal{N}(U,1-U),
\quad    \frac{d_n(\sigma_n)-\frac{n}{2}}{\sqrt{n}} \xrightarrow[n\to\infty]{d} \mathcal{N}(0,6), \quad      \frac{d_k(\sigma_n)-\frac{n}{2}}{\sqrt{n}} \xrightarrow[n\to\infty]{d} \mathcal{N}(0,6),
\end{equation*}
where $U$ is a uniform random  variable on $[0,1]$.
\end{proposition}
}
{
 Note that  $d_k$ is a local statistic for fixed $k$ but it is not the case for $d_n$. 
The uniform case is already studied by \cite{grerk2019study}.  
The problem for $d_n$ is that for any  $2<k<n$, $\varepsilon_{n}(d_n) = n-1$ since $d_n(Id_n)=0$ and $d_n( (n,1))= n-1$ and thus we cannot apply directly our previous approach. The idea  of the proof is the following.  If we condition on the event $$E_n= \{T^1,T^2,\dots,T^n \text{ do not change } \sigma_n(n)\}  ,$$ then $d_n$ changes at most by $2$   every time we apply $T$ and one concludes easily since  $$\mathbb{P}({E_n}) \geq 1-2\frac{\mathbb{E}(\#(\sigma_n))}{n}.$$}
\subsection{A lower bound for fluctuations}
\paragraph{}
{
 For some statistics, one can obtain a better lower bound by using a different way to go from $\sigma_{Ew,0,n}$ to $\sigma_n$.  Unlike the previous examples, the control of the error may depend on the statistic. Our first example is the longest increasing subsequence. We give a lower bound for the fluctuations for a conjugation invariant random permutation. Using this inverse walk one can obtain the following results.} 

\begin{proposition} \label{proposition_2inf}
If 
\normalfont{(\hyperref[hinv1]{${\mathcal{H}}^\p_{inv,\frac{3}{2}}$})}
is satisfied,
then  for any $k\geq 1$, for any $s_{1},\dots,s_{k}\in \mathbb{R}$,
\begin{align*}
\limsup_{n\to \infty}\mathbb{P}\left(\forall i\leq k',\frac{\lambda_i(\sigma_n)-2\sqrt{n}}{{n}^{\frac1 6}}\leq s_i\right)\leq F_2(s_1,s_2,\dots,s_{k}).\end{align*}
In particular,
\begin{align*}
\limsup_{n\to \infty}\mathbb{P}\left(\frac{\mathrm{LIS}(\sigma_n)-2\sqrt{n}}{{n}^\frac{1}{6}}\leq s\right)\leq F_2(s).\end{align*}

\end{proposition}  
To do so, we define a new Markov operator.
Let $\sigma \in \s^0, \lambda \in \mathbb{Y}_n$ and $i\in \{1,\dots,n\}$\footnote{ We recall that $\s^0$ is the set of cyclic permutations.}. We define 
$\mathfrak{T}_{i,\lambda}(\sigma):= \left(\sigma^{\lambda_{1}+1}(i),\dots,\sigma^{\lambda_{1}+\lambda_{2}}(i)\right) \dots\left(\sigma^{\sum_{j=1}^{\ell(\lambda)-1} \lambda_{j}}(i),\dots,\sigma^{n}(i)\right).$
Now  let $\sigma_n$ be a conjugation invariant random permutation and 
     let  $T_{\sigma_n}$ be the Markov operator defined on $\mathfrak{S}^0_n$ as follows. Starting from $\sigma\in\s^0$, choose  $i$ uniformly in $\{1,\dots,n\}$ and $\lambda$ randomly according to the distribution of $\hat{\lambda}(\sigma_n)$\footnote{$\hat{\lambda}(\sigma)$ is the  cycle structure of $\sigma$.} and then 
$T_{\sigma_n}(\sigma)$ returns $ \mathfrak{T}_{i,\lambda}(\sigma).$\footnote{Here we define a different Markov operator for every distribution.}
For example, the transition probabilities of $T_{\sigma_{unif,3}}$.  are shown in Figure~\ref{graph_new_T}. 
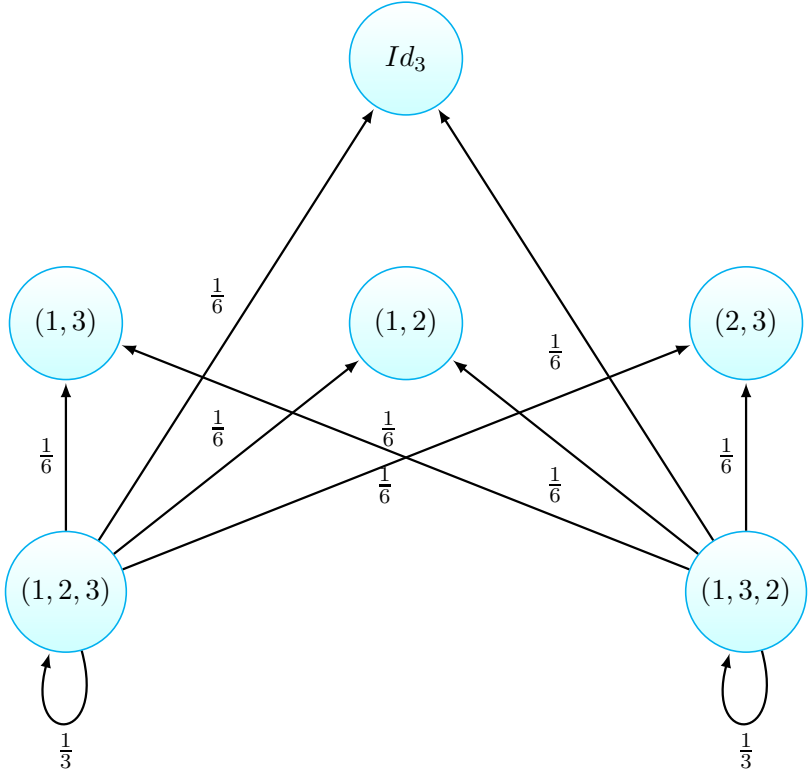
\begin{figure}
\centering
\begin{tikzpicture}
[-latex,auto ,
semithick,
state/.style ={ circle,top color =white, bottom color = processblue!20,
draw,processblue, text= black, minimum width =1.5cm}]
    \node[state] (s1)  {$Id_3$};
    \node[state, below=2cm of s1] (t1) {$(1,2)$};
    \node[state, right=3cm of t1] (t2) {$(2,3)$};
    \node[state, left=3cm of t1] (t3) {$(1,3)$};
    \node[state, below=2cm  of t3] (c1) {$(1,2,3)$};
    \node[state, below= 2cm of t2] (c2) {$(1,3,2)$};
        \draw[every loop,
        line width=0.3mm,
        auto=left,
        >=latex,
        ]
            (c1) edge[]  node {$\frac{1}{6}$} (s1)
             (c2) edge[]  node {$\frac{1}{6}$} (s1)
             
                 (c1) edge[]  node {$\frac{1}{6}$} (t1)
                 (c1) edge[]  node {$\frac{1}{6}$} (t2)
                 (c1) edge[]  node {$\frac{1}{6}$} (t3)             
                 (c2) edge[]  node {$\frac{1}{6}$} (t1)
                 (c2) edge[]  node {$\frac{1}{6}$} (t2)
                 (c2) edge[]  node {$\frac{1}{6}$} (t3)
                 (c2) edge[loop below]  node {$\frac{1}{3}$} (c2)
                 (c1) edge[loop below]  node {$\frac{1}{3}$} (c1);
    \end{tikzpicture}
    \caption{The transition probabilities of $T_{\sigma_{unif,3}}$}
        \label{graph_new_T}

\end{figure}
{
By construction, $\hat{\lambda}( T_{\sigma_n}(\sigma))=\lambda$
and thus, for any cyclic  permutation $\sigma \in \mathfrak{S}_n^0$,
$$\hat{\lambda}(T_{\sigma_n}(\sigma))\overset{d}=\hat{\lambda}(\sigma_n).$$ This yields,
 $$ 
 \hat{\lambda}(T_{\sigma_{n}}(\sigma_{Ew,0,n}))\overset{d}=\hat{\lambda}(\sigma_n).
 $$ 
Finally, since the construction depends only on the cycle structure,  
 $T_{\sigma_{n}}(\sigma_{Ew,0,n})$ is conjugation invariant and
 \begin{align}\label{MCINV}
 T_{\sigma_{n}}(\sigma_{Ew,0,n})\overset{d}=\sigma_n.
 \end{align}
 
Our main argument is  the following lemma.
\begin{lemma}  \label{lemma22}
For any permutation $\rho\in \s^0$, for any conjugation invariant random permutation $\sigma_n$, for any positive integer $k$, almost surely
\begin{align*}
\mathbb{E}\left( \left(\sum_{i=1}^k \lambda_j(T_{\sigma_n}(\rho))- \lambda_j(\rho)\right)_{-} \middle|\#(T_{\sigma_n}(\rho)) \right) &\leq \frac{\#(T_{\sigma_n}(\rho))}{n} 
\sum_{j=1}^k \lambda_i(\rho)
\\ &\overset{d}=  \frac{\#{(\sigma_n)}}{n} 
\sum_{i=1}^k \lambda_i(\rho).
\end{align*}
\end{lemma} \begin{proof}
Let $i_1<i_2<\dots<i_{\sum_{i=1}^k \lambda_i(\rho)}$ such that 
 $\left\{i_1,i_2,\dots<i_{\sum_{i=1}^k \lambda_i(\rho)} \right\} \subset \mathfrak{I}_{k}(\rho)$. We have then for any permutation $\rho'$,
 $$\left\{i_1,i_2,\dots<i_{\sum_{i=1}^k \lambda_i(\rho)} \right\} \cap \{i, \rho'(i)=\rho(i)\} \subset\mathfrak{I}_{k}(\rho')$$ and then
 \begin{align}
      \left(\sum_{j=1}^k \lambda_j(\rho')- \lambda_j(\rho)\right)_{-} \leq \mathrm{card} \left\{j\leq \sum_{i=1}^k  \lambda_i(\rho); \rho(i_j)\neq\rho'(i_j) \right\} .
 \end{align}
 Consequently, almost surely 
\begin{align*}
 \mathbb{E}\left( \left(\sum_{j=1}^k \lambda_j(T_{\sigma_n}(\rho))- \lambda_j(\rho)\right)_{-} \middle|\#(T_{\sigma_n}(\rho)) \right) &\leq 
\sum_{j=1}^{ \sum_{i=1}^k  \lambda_i(\rho)} \mathbb{E} ( \mathbbm{1}_{\rho(j)\neq \rho'(j)}  |\#(T_{\sigma_n}(\rho)))
\\& =  
\sum_{i=1}^k \lambda_i(\rho) \frac{\#(T_{\sigma_n}(\rho))}{n}.
\end{align*}
\end{proof}
 \begin{proof}[Proof of Proposition~\ref{proposition_2inf}]
For any $\varepsilon>0$ there exists $n_0$ such that $$\p\left( \sum_{i=1}^k \lambda_i(\sigma_{Ew,0,n}) < 9k\sqrt{n} \right) \geq \sqrt{1-\varepsilon}$$ and by hypothesis for any $\varepsilon'>0$ there exist $n_1>$ such that for any $n>n_1$
$$\mathbb{P}\left( \#(\sigma_n)< \varepsilon' \frac{n^\frac 2 3}{9k} \right) > \sqrt{1-\varepsilon}.$$
Consequently,
\begin{align*}
\mathbb{P} \left(   \frac{ \mathbb{E}\left( \left(\sum_{j=1}^k \lambda_j(T_{\sigma_n}(\sigma_{Ew,0,n}))- \lambda_j(\sigma_{Ew,0,n})\right)_{-} \middle|\#(T_{\sigma_n}(\sigma_{Ew,0,n})) \right)}{n^\frac 16} <\varepsilon'\right) > 1-\varepsilon.
\end{align*}
This yields
\begin{align*}
     \frac{ \left(\sum_{j=1}^k \lambda_j(T_{\sigma_n}(\sigma_{Ew,0,n}))- \lambda_j(\sigma_{Ew,0,n})\right)_{-} }{n^\frac 16}\xrightarrow[n\to\infty]{\mathbb{P}}0,
\end{align*}
 which concludes the proof since $T_{\sigma_n}(\sigma_{Ew,0,n})\overset d= \sigma_n$.
\end{proof}
\subsection{Proof of Theorems~\ref{the1} and of Proposition \ref{Airyens}} \label{proof1}
Since Theorems~\ref{the1} is the particular case  $k=1$ of Proposition \ref{Airyens}, we will prove only Proposition \ref{Airyens}. Moreover \normalfont{(\hyperref[hinv1]{${\mathcal{H}}^\p_{inv,\frac{3}{2}}$})} implies clearly \eqref{condition_bizarre} and consequently, the first bound of  Proposition~\ref{Airyens} is a direct application of Proposition~\ref{proposition_2inf}. So it is sufficient to prove that under \eqref{hinv} and  \eqref{condition_bizarre}, we have 
\begin{align} \label{borne_sup_1.5}
\liminf_{n\to \infty}\mathbb{P}\left(\forall i\leq k',\frac{\lambda_i(\sigma_n)-2\sqrt{n}}{{n}^{\frac1 6}}\leq s_i\right)\geq F_2(s_1,s_2,\dots,s_{k}).\end{align}
\begin{proof}[Sketch of proof] We will not go trough all the details since we have already presented similar techniques many times. The idea is to  modify the random walk associated to $T$ as following.  Given $1\leq j\leq n-1$, we define 
$\hat{T}_j$ the Markov operator  as following. $\hat{T}_j(\sigma)$ is a permutation chosen uniformly at random among the permutations obtained by merging all cycles of length less than $j$ to (one of) the biggest cycles of $\sigma$ to obtain a permutation with cycles of length more than $j$.  Since this construction depends only on the cycle structure, under \eqref{hinv}, $\hat{T}_j(\sigma_n)$ is conjugation invariant. Therefore $T^n(\hat{T}_j(\sigma_n))$ is distributed according to $Ew(0)$. Similarly to the previous proofs, we have
 \begin{align*}
\mathbb{E}\left( \left(\sum_{i=1}^k \lambda_j(T^n(\hat{T}_j(\sigma_n)))- \lambda_j(\hat{T}_j(\sigma_n)))\right)_{-} \middle|\#(\hat{T}_j(\sigma_n)) \right) &\leq \frac{\#(\hat{T}_j(\sigma_n))}{j} 
\sum_{i=1}^k \lambda_i(\hat{T}_j(\sigma_n)).
\end{align*}
Let $(j_n)_{n>1}$ be such that
$$\frac{1}{{n}^{\frac1 6}}
\left(\left(\sum_{k=1}^{j_n} \#_k(\sigma_n)\right) + \frac{\sqrt{n}}{j_n}\sum_{k=j_n+1}^n \#_k(\sigma_n)\right)\xrightarrow[n\to\infty]{\mathbb{P}}0.$$
We have then $(T_{j_n}(\sigma_n))_{n\geq 1} $ satisfies \normalfont{(\hyperref[hinv1]{${\mathcal{H}}^\p_{inv,6}$})},
$$\frac{\sum_{i=1}^k \lambda_i(\hat{T}_{j_n}(\sigma_n))}{\sqrt{n}}\xrightarrow[n\to\infty]{\mathbb{P}} 2k $$ and 
$$\frac{\mathbb{E}\left( \left(\sum_{i=1}^k \lambda_j(T^n(\hat{T}_j(\sigma_n)))- \lambda_j(\hat{T}_j(\sigma_n)))\right)_{-} \middle|\#(\hat{T}_j(\sigma_n)) \right)}{n^\frac 16} \xrightarrow[n\to\infty]{\mathbb{P}} 0. $$ 
This yields \eqref{borne_sup_1.5}.
\end{proof}}
\subsection{Lower bound for the longest increasing subsequence}
\begin{proposition} \label{lower LIS}
If $(\sigma_n)_{n\geq1}$ is conjugation invariant 
then 
 for any $\varepsilon>0$,
\begin{align*}
    \mathbb{P}\left(\mathrm{LIS}(\sigma_n)>(2\sqrt{13}-6-\varepsilon)\sqrt{n}\right)\xrightarrow[n\to\infty]{} 1.
\end{align*}
This yields the following lower bound 
$$ 
\liminf_{n\to\infty}\frac{\mathbb{E}(\mathrm{LIS}(\sigma_n))}{\sqrt{n}}\geq 2\sqrt{13}-6\simeq1.21\ldots
$$
 \end{proposition}
 Motivated by a conjecture of \cite{MR3509473}, the author tried  in  a previous work    to prove an asymptotic  lower bound  on the expectation of the longest increasing subsequence of a conjugation invariant random permutation without cycle conditions.  In particular,   under the same hypothesis, it was proved in \citep{kam2} that
\begin{align} \label{out_of_date_born}
 \liminf_{n\to\infty}\frac{\mathbb{E}(\mathrm{LIS}(\sigma_n))}{\sqrt{n}}\geq 2\sqrt{\theta} \simeq 0.564\dots,
\end{align}
where $\theta$ is the unique solution of $G(2\sqrt{x})=\frac{2+x}{12}$,
\begin{align}
\nonumber
 G:= [0,2]&\to\left[0,\frac{1}{2}\right]
 \\x &\mapsto
 \int_{-1}^{1}\left(\Omega(s)- \left|s+\frac{x}{2}\right| - \frac{x}{2}\right)_+\mathrm{d} s,
 \label{e2}
\end{align}
and
\begin{align*}
\Omega(s):=\begin{cases}
\frac{2}{\pi}(s\arcsin({s})+\sqrt{1-s^2}) & \text{ if } |s|<1 \\ 
|s| & \text{ if } |s|\geq 1 
\end{cases}.
\end{align*}
\begin{proof}[Sketch of the proof of Proposition~\ref{lower LIS}]
The proof is an adaptation of the proof of \citep[Thm~1]{kam2} 
Before we start, let  $$\theta':=4-\sqrt{13}\; \text{ and } \;\theta'':=2(1-\theta')=2\sqrt{6\theta'-2}=2\sqrt{13}-6=1.21\dots.$$ 

In this proof, we use the following convention. Let $A,B \subset \s$ and $f:\s\to\mathbb{R}$. If $\p(\sigma_n\in A)=0$, we assign $ \p(\sigma_n \in B|  \sigma_n \in A )=0$ and $\E(f(\sigma_n)|\sigma_n \in A)=0$.
\paragraph{} We have 
\begin{align*}
\E(\mathrm{LIS}(\sigma_n))&=\E\left(LIS(\sigma_n)\middle|\#_1(\sigma_n)<\theta''\sqrt{n}\right)\p\left(\#_1(\sigma_n)<\theta''\sqrt{n}\right)\\&+
\E\left(\mathrm{LIS}(\sigma_n)\middle|\#_1(\sigma_n)\geq \theta''\sqrt{n}\right)\p\left(\#_1(\sigma_n)\geq\theta''\sqrt{n}\right)
\\&\geq E\left(LIS(\sigma_n)\middle|\#_1(\sigma_n)<\theta''\sqrt{n}\right)\p\left(\#_1(\sigma_n)<\theta''\sqrt{n}\right)\\&+
\theta''\sqrt{n}\p\left(\#_1(\sigma_n)\geq\theta''\sqrt{n}\right)
.
\end{align*}
Since the condition on the fixed points is  conjugation invariant,  it is sufficient to prove this result in the case where almost surely $\#_1(\sigma_n)<\theta''\sqrt{n}$. 
Using the same argument  and since  the condition on the number of cycles  is  conjugation invariant,  it is sufficient to prove this result in  the 
two particular cases. 
 
\begin{itemize}
    \item If almost surely $\#(\sigma_n)>n\theta'.$
    \\ 
We recall that      \begin{equation*}
      \#_1(\sigma^2)\geq 6\#(\sigma)-3 \#_1(\sigma)-2n.
  \end{equation*}
Consequently, under the condition $\#_1(\sigma_n) < \theta'' \sqrt{n}$,  almost surely,
\begin{equation*}
\#_1(\sigma^2_n)>n(6\theta'-2)-3\theta'\sqrt{n}.    
\end{equation*}
 We can then conclude by \cite[Proposition~{15}]{kam2} that 
 \begin{align*}
 \liminf_{n\to\infty}\frac{\mathbb{E}\left(\mathrm{LIS}(\sigma_n)\right)} {\sqrt{n(6\theta'-2)-3\theta'\sqrt{n}}}\geq2. \end{align*}
Thus,\begin{align*}
 \liminf_{n\to\infty}\frac{\mathbb{E}\left(\mathrm{LIS}(\sigma_n)\right)} {\sqrt{n}}
 \geq2\sqrt{6\theta'-2}=\theta''. \end{align*}

\item If almost surely $\#(\sigma_n)\leq n\theta'.$
Using Lemma~\ref{lemma22} for $k=1$, we obtain that for any $\varepsilon,\varepsilon'>0$ there exists $n_0$ such that for any $n>n_0$, for any conjugation invariant random permutation $\sigma_n$ such that almost surely $\#(\sigma_n)\leq n\theta'$,
$$\p(\mathrm{LIS}(\sigma_n)> 2\sqrt{n}(1-\theta'-\varepsilon)  )>1-\varepsilon'.$$
Consequently, 
$$  \liminf_{n\to\infty}\frac{\mathbb{E}\left(\mathrm{LIS}(\sigma_n)\right)} {\sqrt{n}}
 \geq 2(1-\theta')=\theta''.$$
\end{itemize}
This concludes the proof.

\end{proof}

\section{Other groups}

\subsection{General idea and main results}
\paragraph{}
{
The same technique of proof we presented in  Section~\ref{sec:1} can be applied to other sets having a similar structure to the symmetric group. We will give applications in the next subsection. In general, one can apply the same techniques when there exists a "nice" sequence of \underline{undirected  graphs} $G:=(G_n=(V_n,E_n))_{n\geq 1}$\footnote{We use the usual notations i.e.  $V_n$ is the set of vertices and $E_n$ is the set of edges. } such that \footnote{We use $\sqcup$ to denote disjoint union.}. 
\begin{align}
    \label{HGra1} \forall n\geq 1, \,  G_n \text{ is locally finite.}  
    \end{align}
\begin{align} 
    \forall n\geq 1,\text{ there exists a countable set $I_n$ and finite sets $(V^i_n)_{i\in I_n}$ such that  }  \label{HGra2}  V_n=\sqcup_{i\in I_n}
    V^i_n.\end{align}
    For any $n\geq 1$, for any $i,j\in I_n$, for any $\sigma_1,\sigma_2 \in V^i_n$, 
    \begin{align}
    \label{HGraIN}
             \mathrm{card}(\{ \sigma'\in V^j_n; (\sigma',\sigma_1)\in E_n  \}) = \mathrm{card}(\{ \sigma'\in V^j_n; (\sigma',\sigma_2)\in E_n  \})=:\mathrm{e}_{j,i}.
 \end{align}
 i.e. the number of neighbors in $V^j_n$ of any element of $V^i_n$ only depends on $(i,j)$; we denote it by $\mathrm{e}_{i,j}$.
 We denote by  $\widetilde{E_n} :=\{(i,j) \in I^2_n; \mathrm{e}_{i,j} >0 \}$ and  by $\widetilde{G_n}:=(I_n, \widetilde{E_n})$ the  classes graph. We need  moreover in the sequel of this Section~\ref{sec:1}  that
\begin{align}
   \label{HGraCon}  
  \forall n\geq 1, \,     \text{the classes graph } \widetilde{G_n} \text{ is connected.}
\end{align}
\\ 
\underline{In the sequel of this section, we assume \eqref{HGra1}-- \eqref{HGraCon}.}
\\ 
For example, if $G_n$ is the  Cayley graph of the symmetric group generated by transpositions we have 
\begin{itemize}
\item  $V_n=\s$
\item $E_n= \{(\sigma,\sigma\circ(i,j)) ; \sigma \in \s ,\, i\neq j  \}$
 \item $I_n$=$\mathbb{Y}_n$ (the set of Young diagrams of size $n$).
 \item $V_n^i=\{\sigma \in \s; \hat{\lambda}(\sigma)=i\}$,
 \item $\widetilde{E_n}$ the set of couples of Young diagrams such that one can obtain one from the other by concatenating two arrows. For example, for $n=4$, we obtain the classes graph in Figure \ref{figtrclass2Z}.
 \begin{figure}
\centering
\begin{tikzpicture}
[- ,auto  , 
semithick , scale=0.6, every node/.style={transform shape}, state/.style={circle,inner sep=2pt}]
    \node[state] (s1)  {\yng(1,1,1,1)};
    \node[state, right=1cm of s1] (t1) {\yng(2,1,1)};
    \node[state, right=1cm of t1, yshift=2cm] (t2) {$\yng(2,2)$};
	\node[state, right=1cm of t1,  yshift=-2cm] (c1) {\yng(3,1)};
	\node[state, right=1cm of t2,  yshift=-2cm] (c2) {\yng(4)};
        \draw[
        ]
            (s1) edge[]  node  {} (t1)
            (t1) edge[]  node  {} (t2)
            (t1) edge[]  node  {} (c1)
            (t2) edge[]  node  {} (c2)
            (c1) edge[]  node  {} (c2)    ;
 \end{tikzpicture}
    \caption{The  classes graph for the Cayley graph of $\s$ generated by transpositions  for $n=4$}
    \label{figtrclass2Z}
\end{figure}
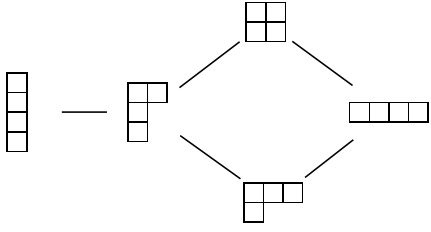

\end{itemize} 

\paragraph{}
With analogy with  Section~\ref{sec:1}, we will now construct a new directed graph for which we will consider the uniform random walk. 
Let $d_{G_n}$ be  the usual graph distance and for  $\sigma\in V_n$, we  denote by  $Class(\sigma)$ the unique $i\in I_n$ such that $j\in V^i_n$. 
\paragraph{} Let $(i^*_n)_{n\geq1} \in \prod_{n\geq1}I_n$ be a "nice" sequence of classes.
We denote by $\underline{d}(\sigma):= \min_{\rho \in V^{i^*_n}_n} d_{G_n}(\sigma,\rho )$. 
The random walk we use to prove universality  will be the uniform random walk on the \underline{directed graph}  $G'_n:=(V_n,E'_n)$ where 
$$E'_n =\{(\sigma_1,\sigma_2) \in E_n  \, ;  \underline{d}(\sigma_2)=\underline{d}(\sigma_1)-1\}  \cup \{(\sigma,\sigma), \sigma\in V_n^{i^*_n} \}.$$ 
Back to the example of  the Cayley graph of the symmetric group generated by transpositions we have 
\begin{itemize}
 \item $Class(\sigma)=\hat{\lambda}(\sigma)$,
 \item $i^*_n=(n,\underline{0})$  is the Young diagram with a unique row of length $n$, 
 \item $(V_n,E'_n)={\mathcal{G}_{\mathfrak{S}_n}}$\footnote{${\mathcal{G}_{\mathfrak{S}_n}}$  is defined in  Section~\ref{sec:1} },
 \item $\underline{d}(\sigma)= \#(\sigma)-1$.
\end{itemize}
\paragraph{}
With analogy with the symmetric group, let  $T_{G'_n}$ be the Markov operator associated to the uniform random walk on $G'_n$, $V_\infty:=\cup_{n\geq1} V_n$ and $f$ be a function defined on $V_\infty$ and having values on some metric space $(F,d_F)$. With analogy with  Section~\ref{sec:1} , for $S\subset V_n$ and $\sigma\in V_n$, let
$$\mathrm{next}(S):=\{\sigma_2;  \sigma_1\in S \, \text{ and } \,
(\sigma_1,\sigma_2) \in E'_n\},$$ 
$$\mathrm{final}(\sigma):=\begin{cases} 
\mathrm{next}^{\underline{d}(\sigma)}( \{\sigma\})  & \text{if }  \underline{d}(\sigma)>1   \\ \{\sigma\} & \text{otherwise} \end{cases}$$ 
and for $i \in I_n$ and $p\geq 1$, we define  
\begin{align*}
     \underline{\varepsilon}_{n,i,p}(f)&:=\left({\sum_{\sigma\in V^i_n}\sum_{\rho\in \mathrm{next}(\{\sigma\})}} \frac{(d_F(f(\sigma),f(\rho)))^p}{\mathrm{card}(V^i_n)\mathrm{card}(\mathrm{next}(\{\sigma\}))}\right)^\frac{1}{p}
    \\ \underline{\varepsilon}_{n,p}(f)&:= \sup_{i \in I_n} \underline{\varepsilon}_{n,i,p}(f)
    \\   \underline{\varepsilon}_{n,i,\infty}(f)&:=\max_{\sigma\in V^i_n} \max_{\rho\in \mathrm{next}(\{\sigma\})} d_F(f(\sigma),f(\rho))
      \\ \underline{\varepsilon}_{n,\infty}(f)&:= \sup_{i \in I_n} \underline{\varepsilon}_{n,i,\infty}(f)
           \\  \underline{\varepsilon}'_{n,i,p}(f)&:=\left({\sum_{\sigma\in V^i_n}\sum_{\rho\in \mathrm{final}(\sigma)}} \frac{(d_F(f(\sigma),f(\rho)))^p}{\mathrm{card}(V^i_n)\mathrm{card}(\mathrm{final}(\sigma))}\right)^\frac{1}{p}
    \\
        \underline{\varepsilon}'_{n,i,\infty}(f)&:=\max_{\sigma\in V^i_n} \max_{\rho\in \mathrm{final}(\sigma)} d_F(f(\sigma),f(\rho)).
    \end{align*}
    Finally, let  $(\sigma_n)_{n\geq1}$ be a sequence of  random  variables such that $\sigma_n$ is supported on $V_n$. We say   that 
$\sigma_n$ is $G_n$ invariant (with respect to the partition $\{V^i_n\}_{i\in I_n}$)\footnote{we omit this precision  when it is clear from the context.}
if  for any  $i\in I_n$ and any $\sigma,\rho \in V_n^{i}$  
\begin{align*} 
 \mathbb{P}(\sigma_n=\sigma)=\mathbb{P}(\sigma_n=\rho),\end{align*}
and we say that  
$(\sigma_n)_{n\geq1}$ is  $G$-invariant if $\sigma_n$ is $G_n$-invariant $\forall n\geq 1 $. 
 \begin{definition}
For $\alpha > 0$ and $p\in[1,\infty]$, we say that   $(\sigma_n)_{n\geq 1}$ satisfies   $\mathcal{H}_{G-inv,\alpha}^\p$  if   
\begin{align} \label{Hhinv1}
 (\sigma_n)_{n\geq 1}  \text { is G-invariant  and } \quad  
\tag{$\mathcal{H}_{G-inv,\alpha}^\p$}
    \frac{ \underline{d}(\sigma_n)}{n^ {\frac1\alpha}} \xrightarrow[n\to\infty]{\p}0,
\end{align}
we say that it satisfies 
  $\mathcal{H}_{G-inv,\alpha}^{\mathbb{L}^p}$  if   
\begin{align} \label{Hhinv1p}
 (\sigma_n)_{n\geq 1}  \text { is G-invariant  and } \quad  
\tag{$\mathcal{H}_{G-inv,\alpha}^{\mathbb{L}^p}$}
    \frac{ \underline{d}(\sigma_n)}{n^ {\frac1\alpha}} \xrightarrow[n\to\infty]{\mathbb{L}^p}0.
\end{align}
\end{definition}
Interesting results can be obtained  if the graph satisfies an additional symmetry property: 
\begin{itemize}
    \item[\ ]   For any $\sigma_1 \in V_n$,  for any $\sigma_2,\sigma_3  \in \mathrm{final}(\sigma_1)$,  the number of paths in $G'_n$  of length $\underline{d}(\sigma)$ from $\sigma_1$ to $\sigma_2$  is equal to that from $\sigma_1$ to $\sigma_3$ i.e. $A_{G'_n}$ the adjacency matrix of $G'_n$ satisfies the following:    \begin{align} 
    \forall \sigma_1\in \s, \exists \, c_{\sigma_1}\in \mathbb{N} \text{ such that } \forall \rho\in\s, \, \label{symeq}{A_{G'_n}^{\underline{d}(\sigma)}}(\sigma_1,\rho)=c_{\sigma_1}\mathbbm{1}_{\rho\in \mathrm{final}(\sigma_1)}.\end{align}
\end{itemize} }

{
In particular, we have the following:
\begin{lemma}
Under \eqref{HGra1}--\eqref{symeq}, for any  $(\sigma_n)_{n\geq 1}$  $G$-invariant, for any $\p\in[1,\infty[$,
$$\mathbb{E} \left(\left(d_F\left(f(\sigma_n),f\left(T_{G'_n}^{\underline{d}(\sigma_n)}(\sigma_n)\right)\right)\right)^p\right) =\mathbb{E} ((\underline{\varepsilon}'_{n,Class(\sigma_n),p})^p).$$ 
\end{lemma}
\begin{proof}
For any random variable $\sigma_n$, we have 
\begin{align*} 
 \mathbb{E} \left(\left(d_F\left(f(\sigma_n),f\left(T_{G'_n}^{\underline{d}(\sigma_n)}(\sigma_n)\right)\right)\right)^p\right) &=\mathbb{E}\left(\mathbb{E} \left(\left(d_F\left(f(\sigma_n),f\left(T_{G'_n}^{\underline{d}(\sigma_n)}(\sigma_n)\right)\right)\right)^p\middle|\sigma_n\right)\right)
\\&= \sum_{i\in I_n} \sum_{\sigma \in V^i_n}\mathbb{P}(\sigma_n=\sigma) \mathbb{E} \left(\left(d_F\left(f(\sigma_n),f\left(T_{G'_n}^{\underline{d}(\sigma_n)}(\sigma_n)\right)\right)\right)^p\middle|\sigma_n=\sigma\right). 
\end{align*}
If $(\sigma_n)_{n\geq 1}$ is  $G$-invariant, then $\mathbb{P}(\sigma_n=\sigma)= \frac{1}{\mathrm{card}(Class(\sigma))} \mathbb{P}(Class(\sigma_n)Class(\sigma))$. 
Moreover, under  \eqref{symeq}, 
$$
\mathbb{E} \left(\left(d_F(f(\sigma_n),f\left(T_{G'_n}^{\underline{d}(\sigma_n)}(\sigma_n)\right)\right)^p\middle|\sigma_n=\sigma\right)=\frac{\sum_{\rho \in \mathrm{final}(\sigma)}(d_F(f(\sigma),f(\rho)))^p }{\mathrm{card}(\mathrm{final}(\sigma))}.
$$
Consequently, one can conclude since 
\begin{align*}
\mathbb{E} ((\underline{\varepsilon}'_{n,Class(\sigma_n),p})^p) =& \mathbb{E} \left(\mathbb{E} ((\underline{\varepsilon}'_{n,Class(\sigma_n),p})^p) \middle| Class(\sigma_n)\right)\\&= \sum_{i\in I_n} \mathbb{P}(Class(\sigma_n)=i)(\underline{\varepsilon}'_{n,i,p})^p
\\&= \sum_{i\in I_n} \mathbb{P}(Class(\sigma_n)=i) {\sum_{\sigma\in V^i_n}\sum_{\rho\in \mathrm{final}(\sigma)}} \frac{(d_F(f(\sigma),f(\rho)))^p}{\mathrm{card}(V^i_n)\mathrm{card}(\mathrm{final}(\sigma))}.
\end{align*}
\end{proof}
Similarly, one can prove the following.
\begin{lemma}
Under \eqref{HGra1}--\eqref{HGraCon}, $(\sigma_n)_{n\geq 1}$ is $G$-invariant, for $n\geq1$, for any $\p\in[1,\infty[$,
$$\mathbb{E} \left(\left(d_F(f(\sigma_n),f(T_{G'_n})(\sigma_n))\right)^p\right) =\mathbb{E} ((\underline{\varepsilon}_{n,Class(\sigma_n),p})^p).$$ 
\end{lemma}
This gives as a universality result. 
\begin{theorem}  Assume  that \eqref{HGra1}--\eqref{HGraCon} and that  $(\sigma_n)_{n\geq 1}$ and $(\sigma_{ref,n})_{n\geq 1}$ 
are  $G$-invariant.
Suppose that there exists some  deterministic $x\in F$ and $p \in [1,\infty[$   such that 
\begin{align*}
f(\sigma_{ref,n}) \xrightarrow[n\to\infty]{\mathbb{P}} x \quad (\text{ resp. } f(\sigma_{ref,n}) \xrightarrow[n\to\infty]{\mathbb{L}^p} x \text{}),
\end{align*} 
\begin{align} \label{secondcont1}
\underline{\varepsilon}_{n,Class(\sigma_{ref,n}),\infty}(f) \xrightarrow[n\to\infty]{\mathbb{P}} 0  \quad 
(\text{ resp. }
\underline{\varepsilon}_{n,Class(\sigma_{ref,n}),\infty} (f)&\xrightarrow[n\to\infty]{\mathbb{L}^p} 0) \end{align}
and 
\begin{align}  \label{secondcont2} 
\underline{\varepsilon}_{n,Class(\sigma_n),\infty} (f)\xrightarrow[n\to\infty]{\mathbb{P}} 0
\quad 
(\text{ resp. }
\underline{\varepsilon}_{n,Class(\sigma_n),\infty} (f)\xrightarrow[n\to\infty]{\mathbb{L}^p} 0).
    \end{align}
Then
\begin{align*}
f(\sigma_{n}) \xrightarrow[n\to\infty]{\mathbb{P}} x  \quad (\text{resp. } f(\sigma_{n}) \xrightarrow[n\to\infty]{\mathbb{L}^p} x).
\end{align*}
Moreover, under \eqref{symeq},  \eqref{secondcont1} and \eqref{secondcont2} can be replaced by 
\begin{align*} 
\underline{\varepsilon}'_{n,Class(\sigma_{ref,n}),1}(f) \xrightarrow[n\to\infty]{\mathbb{P}} 0  \quad 
(\text{ resp. }
\underline{\varepsilon}'_{n,Class(\sigma_{ref,n}),p}(f) &\xrightarrow[n\to\infty]{\mathbb{L}^p} 0) \end{align*}
and 
\begin{align*}
\underline{\varepsilon}'_{n,Class(\sigma_n),1}(f) \xrightarrow[n\to\infty]{\mathbb{P}} 0
\quad 
(\text{ resp. }
\underline{\varepsilon}'_{n,Class(\sigma_n),p}(f) \xrightarrow[n\to\infty]{\mathbb{L}^p} 0).
    \end{align*}
\end{theorem}
\begin{proof}[Idea of the proof] 
The proof is identical to that  of theorems~\ref{main_rest1} and \ref{universality_RD}. Indeed,  \eqref{HGraIN}
 guarantees   that under  the $G$-invariance, 
for any $n\geq 1$,  $T_{G'_n}(\sigma_n)$ is $G_n$ invariant and  by construction almost surely 
$$\underline{d}(T_{G'_n}(\sigma_n))=\max(0, \underline{d}(\sigma_n)-1).$$  Consequently, by induction, $T^{\underline{d}(\sigma_n)}_{G'_n}(\sigma_n)$ is distributed according to the uniform  distribution on $V^{i^*_n}_n$ and  almost surely 
$$d_F(f(T^{\underline{d}(\sigma_n)}_{G'_n}(\sigma_n)),f(\sigma_n)) \leq \underline{\varepsilon}_{n,class(\sigma_n),\infty}(f).$$
\end{proof}
Similarly to Remark~\ref{remarque_RD}, 
by the triangle inequality and using that the arithmetic mean is smaller than the $p$-mean, we have \footnote{There is here a notation abuse. Since $\underline{d}$ is constant in any class, we denote by $\underline{d}(k)$, $\underline{d}{(\sigma)}$ for some $\sigma\in k$. }
$$(\underline{\varepsilon}_{n,k,p}(f))^p \leq \sum_{i=1}^{\underline{d}(k)} \max_{j; \underline{d}(j)=i} ({\underline{\varepsilon}^p_{n,j,p}(f)}) \leq \underline{d}(k)\underline{\varepsilon}^p_{n,p}(f). $$ 
Consequently, if there exists $\alpha>0$ such that 
$$\underline{\varepsilon}^p_{n,p}(f)=O\left(\frac 1{n^\frac 1\alpha}\right),$$
then one can obtain \eqref{secondcont1}  and \eqref{secondcont2} for the  equivalent classes  of \eqref{Hhinv1}  (resp.\eqref{Hhinv1p}).
}
\subsection{Some examples of finite graphs}
In general, Cayley graphs are  good candidates. An interesting case is when   there exists $(i^*_n)_{n\geq1} \in \prod_{n\geq1}I_n$   such that $$\frac{1}{\mathrm{card}(V_n)} \sum_{\sigma\in V_n} \min_{\sigma' \in V^{i^*_n}_n} d_{G_n}(\sigma,\sigma' )=o\left(\max_{\sigma_1,\sigma_2 \in V_n} d_{G_n}(\sigma_1,\sigma_2 )\right), $$
in this case, the comparison with the uniform distribution can be done for reasonable statistics.  
The first four examples we give are different ways to apply our results to the symmetric group. The other four  examples are different graphs.   Our eight examples satisfy  \eqref{HGra1}-- \eqref{symeq}. In the first two examples we will give in details the different objects, for the other we will give only  $G_n$, $I_n$ $V^i_n$ and $i^*_n$. The others can be obtained easily by applying the definitions.  {\begin{itemize}
 \item The Cayley graph of symmetric group generated by transpositions:  We recall that 
 \begin{itemize}
    \item $(V_n,E'_n)={\mathcal{G}_\mathfrak{S}}_n$, 
\item $I_n$=$\mathbb{Y}_n$,
 \item $V_n^i=\{\sigma \in\s;  \hat{\lambda}(\sigma)=i\}$,
 \item $Class(\sigma)=\hat{\lambda}(\sigma)$,
 \item $i^*_n=(n,\underline{0})$  the Young diagram with a unique row of length $n$, 
 \item $\underline{d}(\sigma)= \#(\sigma)-1$,
\end{itemize}
We have then the following.  
\begin{align*}
    \frac{1}{\mathrm{card}(V_n)} \sum_{\sigma\in V_n} \min_{\sigma' \in V^{i^*}_n} d_{G_n}(\sigma,\sigma* ) &=
    \mathbb{E}(\#(\sigma_{unif,n})-1) \\&=
    \sum_{k=2}^{n} \frac{1}{k} \ =o(n-1)= o\left(\max_{\sigma_1,\sigma_2 \in V_n} d_{G_n}(\sigma_1,\sigma_2 )\right).\end{align*}
 \item  Even permutations: A permutation $\sigma \in\s$ is said to be even if $n-\#(\sigma)$ is even. Cycles of length $3$ are a generator of $\s$. When $n$  is odd, $\s^0$ is a subset of the set of  even permutations. One can choose for example. 
 \begin{itemize}
 \item  $G_n$ the Cayley graph of $\mathfrak{S}_{2n+1}$ generated by cycles of length $3$
\item $I_n$=$\{\lambda\in \mathbb{Y}_{2n+1} ; \ell(\lambda)\equiv 1 \pmod 2 \}$, 
 \item $V_n^i=\{\sigma \in \mathfrak{S}_{2n+1}, \hat{\lambda}(\sigma)=i\}$, 
 \item $Class(\sigma)=\hat{\lambda}(\sigma)$,
 \item $i^*_n=(2n+1,\underline{0})$, 
 \item $\underline{d}(\sigma)= \frac{\#(\sigma)+1}2$.
 \end{itemize}
  \item $\s$ seen as a Coxeter group: Here  we take the right (or the left) Cayley graph generated by transpositions of type $(i,i+1)$. \footnote{Fun fact 2: depending on the  choose of the  right or the left composition,   one can obtain a different universality theorem. The classes are the same but the graph (and consequently error controls)  are different.  }  In this case we have: 
  \begin{itemize}
 \item  $G_n$ the right (or the left) Cayley graph of $\mathfrak{S}_{n}$ generated by  $\{(i,i+1) ; 1\leq i\leq n-1 \}$.

\item  $I_n=\{0,1,\dots,\frac{n(n-1)}{2}\}$, 
 \item $V_n^i=\{\sigma; \mathcal{K}_2(\sigma)=i\}$, where we recall that $\mathcal{K}_2(\sigma)$ is the number of inversions of $\sigma$. 
 \item $Class(\sigma)=\mathcal{K}_2(\sigma)$,
 \item $i^*_n=\ceil{\frac{n^2}{4}}$ , 
 \item $\underline{d}(\sigma)= |\ceil{\frac{n^2}{4}}  -\mathcal{K}_2(\sigma)|$.
 \end{itemize}
For example, $G'_3$ is represented in Figure~\ref{figtrclass5}.
 Corollary~\ref{big_univ_corr} guarantees that
$i^*_n=\ceil{\frac{n^2}{4}}$ is a good candidate if we want to compare with the  uniform distribution. But also it is possible to choose $i^*_n=0$ when looking for  universality results for random permutations close to the identity.  For this graph, the Mallows law with Kendall tau distance  is $G_n$-invariant and one can obtain a first order universality for all local statistics we already studied in the previous sections and for the limiting shape\footnote{We apologize again to the reader because it is not defined yet.}. The second order fails.
\begin{figure}
\centering
\begin{tikzpicture}
[-latex,auto , 
semithick, scale=0.8, every node/.style={transform shape}, state/.style={circle,inner sep=2pt}]
    \node[state] (s1)  {$Id_3$};
    \node[state, below=1cm of s1,xshift=-2cm] (t1) {$(1,2)$};
    \node[state,  below=1cm of s1,xshift=2cm] (t2) {$(2,3)$};
    \node[state, below=1cm  of t1] (c1) {$(1,2,3)$};
    \node[state, below=1cm of t2] (c2) {$(1,3,2)$};
        \node[state, below=1cm of c2,xshift=-2cm] (t3) {$(1,3)$};
        \draw[every loop,
        line width=0.3mm,
        auto=left,
        >=latex,
        ]
            (s1) edge[]  node  {} (t1)
              (s1) edge[]  node {} (t2)
                 (t1) edge[ ]  node {} (c1)
                 (c1) edge[]  node {} (t3)
                 (t2) edge[]  node {} (c2)
                 (c2) edge[]  node {} (t3)
                (t3) edge[loop below]  node {} (t3);
 \end{tikzpicture} \quad \quad \quad 
 \begin{tikzpicture}
[-latex,auto , 
semithick, scale=0.8, every node/.style={transform shape}, state/.style={circle,inner sep=2pt}]
    \node[state] (s1)  {$Id_3$};
    \node[state, below=1cm of s1,xshift=-2cm] (t1) {$(1,2)$};
    \node[state,  below=1cm of s1,xshift=2cm] (t2) {$(2,3)$};
    \node[state, below=1cm  of t1] (c1) {$(1,3,2)$};
    \node[state, below=1cm of t2] (c2) {$(1,2,3)$};
        \node[state, below=1cm of c2,xshift=-2cm] (t3) {$(1,3)$};
        \draw[every loop,
        line width=0.3mm,
        auto=left,
        >=latex,
        ]
            (s1) edge[]  node  {} (t1)
              (s1) edge[]  node {} (t2)
                 (t1) edge[ ]  node {} (c1)
                 (c1) edge[]  node {} (t3)
                 (t2) edge[]  node {} (c2)
                 (c2) edge[]  node {} (t3)
                (t3) edge[loop below]  node {} (t3);
 \end{tikzpicture}
     \caption{$G'_3$ obtained by the transpositions $(1,2)$ and $(2,3)$ }
    \label{figtrclass5}
 \end{figure}
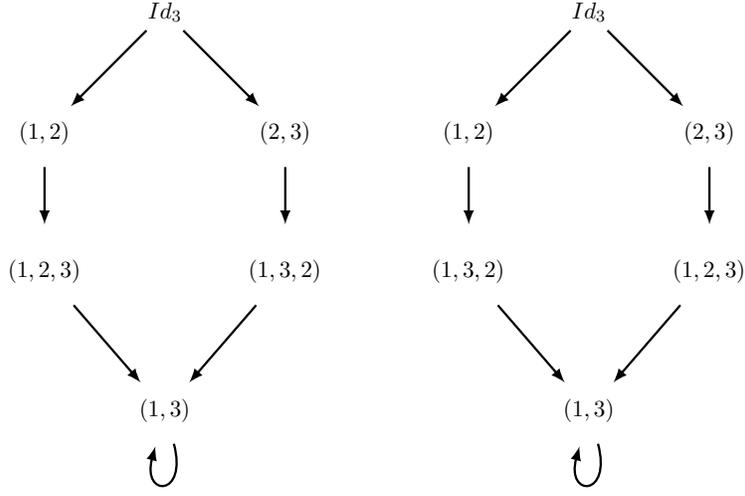 
 \item  Using the same previous graph (same $G_n$) but with only two classes even and odd permutations\footnote{Here, the choice of $i^*_n$ is not important but the reader can take $i^*_n={\text{even}}$.} i.e. $I_n=\{\text{even,\ odd}\}$
 we obtain that,  if  $f(\sigma_n)$ converges  in probability (or $\mathbb{L}^1$) when $\sigma_n$ follows  one of these  three distributions \begin{itemize}
    \item Uniform law of $\s$
    \item Uniform law of even permutations 
    \item Uniform law of odd permutations
\end{itemize} it converges also for  the two others
as soon as 
$$ \frac{\min\left(\sum_{\sigma\in\s,1\leq i<n}d_F(f(\sigma\circ(i,i+1)),f(\sigma)); \sum_{\sigma\in\s,1\leq i<n}d_F(f((i,i+1)\circ \sigma),f(\sigma))\right)}{n!(n-1)} =o(1).$$
\item Another possible  application is the hypercube   ${(\mathbb{Z}/2\mathbb{Z})^{2n}}$.  In this case, we set
\begin{itemize}
    \item $G_n={(\mathbb{Z}/2\mathbb{Z})^{2n}}$
    \item $I_n=\{0,1,\dots,2n\}$
    \item   $V_n^i$  is the set of edges of the graph such that the graph distance from $(0,\dots,0)$ is $i$.
    \item      $i^*_n=n$. 
\end{itemize}
 In this case, 
\begin{align*} 
\frac{1}{\mathrm{card}(V_n)} \sum_{\sigma\in V_n} \min_{\sigma' \in V^{i^*_n}_n} d_{G_n}(\sigma,\sigma* ) &= \frac{\sum_{k=0}^{2n} |\binom{2n}{k} (k-n)|} {4^n}\\&\leq 
\sqrt{\frac{\sum_{k=0}^{2n} \binom{2n}{k}(k-n)^2} {4^n}}
\\&=\sqrt{\frac{n}{2}}\\&=o(n)=  o\left(\max_{\sigma_1,\sigma_2 \in V_n} d_{G_n}(\sigma_1,\sigma_2 )\right).
\end{align*}
 \item   $(\mathbb{Z}/d\mathbb{Z})^{nd}:$ 
Let $\mathcal{R}_n$ be the equivalent relation defined as follows: For any $$x=(x_i)_{1\leq i \leq nd},y=(y_i)_{1\leq i \leq nd} \in  (\mathbb{Z}/d\mathbb{Z})^{nd}, \, x\mathcal{R}_n y \Leftrightarrow  \exists \sigma\in \mathfrak{S}_{nd}, y=(x_{\sigma(i)})_{1\leq i\leq nd}.$$  $\mathcal{R}_n$ define naturally the classes of the vertices.  The central limit theorem guarantees  that  the class $i^*_n$ where we have exactly  $n$ coordinates equal to $k$ for any $k$ in $\mathbb{Z}/d\mathbb{Z}.$ is a good candidate\footnote{Fun fact 3: by choosing fixed and different proportions of every element of $\mathbb{Z}/d\mathbb{Z}$ for $i^*_n$, one can obtain different universality result. }.
\item Let $(H_n)_{n\geq 1}$ be a sequence of  non-commutative and finite groups and $(A_n)_{n\geq 1}$ such that $A_n$ is   a conjugation invariant subset of $H_n$\footnote{i.e. if $\sigma\in A_n$ then $\bar\sigma \subset H_n$.}. 
\begin{itemize}
    \item $G_n$ be the Cayley graph generated by $H_n$.
    \item  $I_n$ is the set of  conjugacy classes 
    \item $V^{i}_n=i$ 
\end{itemize}
In this case, $G$-invariant random variables are conjugation invariant variables. The choice of $i^*_n$ is specific to the choice of $G_n$.   

 \item Dihedral group $\mathbb{D}_{2n}$ \footnote{This is a typical "bad"  Cayley graph since its diameter  is bounded  (equal to $2$) and consequently the universality result is trivial.  } with $n\geq 3$: The Dihedral group $\mathbb{D}_{2n}$ is  defined via its  representation $ < \sigma,\mu |  \sigma ^{2},\mu ^{2},(\mu \sigma )^{n}> $\footnote{This notation is classic to define groups. It  means in our case that  $\mathbb{D}_{2n}$ is isomorphic to  the group generated by $\sigma$ and $\mu$ such that $\sigma^2=\mu^2=(\mu\sigma)^n=1$}. This representation shows that $\mathbb{D}_{2n}$ is a Coxeter group.  For our study, one can admit that $$\mathbb{D}_{2n}=\{\mathrm{s}_0,\dots,\mathrm{s}_{n-1},\mathrm r_0,\dots,\mathrm r_{n-1}\}$$
and 
$$ \mathrm{r}_i\,\mathrm{r}_j = \mathrm{r}_{i+j}, \quad \mathrm{r}_i\,\mathrm{s}_j = \mathrm{s}_{i+j}, \quad \mathrm{s}_i\,\mathrm{r}_j = \mathrm{s}_{i-j}, \quad \mathrm{s}_i\,\mathrm{s}_j = \mathrm{r}_{i-j}.$$
Here, $(i,j)$ are in $\mathbb{Z}/n\mathbb{Z}$.
One can choose either 
\begin{itemize}
    \item $G_n$: the Cayley graph generated  by $\{s_i,0\leq i\leq n\}$,
    \item $I_n= \{ \text{r,s} \}$,
    \item $V^s_n= \{s_i,0\leq i\leq n\}$ is the set of transpositions  and $V^r_n=\{r_i,0\leq i\leq n\}$ is the set of rotations, 
    \item $i^*_n=r$
\end{itemize}
 or   keep the same graph and choose conjugacy classes as classes (as in  the previous examples)\footnote{There are $n+1$ or $n+2$ depending on the parity of $n$.}. In the second  case, we require that $f(\sigma_n)$ convergences for any sequence of rotations and a transposition does not change a lot the statistic.  
 \item Colored permutations: 
A less trivial example is the set   of signed permutations and more generally the set of colored permutations.
Given two positive integers $n$ and $m$,
a colored permutation is a map
$\pi=(\sigma,\phi)$ 
such that $\sigma\in \mathfrak{S}_n$ and $\phi \in \{1,\dots,n\}^{\{1,\dots,m\}} $. A subsequence $ \pi(x_1),\dots,\pi(x_k)$ of $\pi$ is called increasing of length $m(k-1)+p$ if   $\sigma(x_1)<\sigma(x_2)<\dots<\sigma(x_k)$ and  $\phi(x_1)=\phi(x_2)=\dots=\phi(x_k)=p$. We denote by $\mathrm{LIS}(\pi)$ the length of a longest increasing subsequence.
\begin{theorem} \label{yizzzzi}
Let $(\pi_n=(\sigma_n,\phi_n))_{n\geq1}$ be  a sequence of random colored permutations and assume that:
\begin{itemize}
 \item $\sigma_n$ is independent of $\phi_n$,
 \item $\phi_n$ is distributed according to the uniform distribution,
 \item $\sigma_n$ is  conjugation invariant,
 \item $ \frac{\#\sigma_n}{n^\frac{1}{6}} \overset{\mathbb{P}}{\to} 0$.
\end{itemize}
then,
\begin{align}
\mathbb{P}\left(\frac{\mathrm{LIS}(\pi_n)-2\sqrt{nm}}{m^\frac{2}{3}\sqrt[6]{nm}}<s\right)\to F^m_2(s). 
\end{align}
\end{theorem}
\begin{proof}
The uniform case is proved by \cite{1999math......2001B}. To apply our theorem, choose the  graph where two colored permutations are related by an edge if only the first components differ by a transposition i.e.
\begin{itemize}
\item $E_n:= \{ ((\sigma,\phi),(\sigma\circ(i,j),\phi)) ;  i\neq j  \}$,
\item $I_n=\mathbb{Y}_n$,
\item $V^i_n= \{(\sigma,\phi); \bar\sigma=i\}$,
 \item $i^*_n=(n,\underline{0})$.
\end{itemize}
\end{proof} 
\end{itemize}
For  our examples, a trivial example of $G-$invariant elements is the uniform measure, or the uniform measure on a given class. Since $\underline{d}$ is constant in classes,  a natural way to  generalize Ewens measures is the following.
Given $q\in \mathbb{R}_+ $, the  probability measure satisfying  
$$\mathbb{P} (\rho_{G,q,n}=\sigma)=\frac{q^{\underline{d}(\sigma)} }{\sum_{\sigma'\in V_n}q^{\underline{d}(\sigma')} } $$ 
is $G-$invariant and for any statistic $f$ such that $f(\rho_{G,0,n})$ converges,
 one can obtain   a  non-empty universality result around $\rho_{G,0,n}$ since 
 
 $$err_n:=q\mapsto \mathbb{E}(d_F( f(T^{\underline{d}(\rho_{G,q,n})}(\rho_{G,q,n}),f(\rho_{G,q,n}))))  
 $$
 is continuous and $err_{n}(0)=0$. In fact, in the case of permutations, Ewens and Mallows measures with  Kendall tau distance are  particular case of  $\rho_{G,q,n}$. 

\subsection{Infinite case } 
We take now $G_n=G$ an infinite graph. Example of "nice graphs":
\begin{itemize}
    \item The infinite $d$-regular tree $\mathfrak{T}_d$.
    \item  The set of words of a finite alphabet of length $d$.
    \item The free group $\mathcal{F}_d$ with its natural Cayley graph.
    \item The Cayley graph of $\mathcal{B}_d$, the Artin  Braid group.
    \item The Cayley graph of  an infinite and finitely generated  group $H=<x_1,x_2,\dots,x_n>$. 
\end{itemize}
The classes  here are indexed by $\mathbb{N}$ according to the distance to the root (or the identity). 
Let $G$ be such that $$0<   \liminf_{n\to\infty} \frac{\log(card (\{x;\underline{d}(x)=n\}))}{n} = \limsup_{n\to\infty} \frac{\log(card (\{x;\underline{d}(x)=n\}))}{n} = \log(\lambda) <\infty.$$
It is the case for the first three examples. 
Let $f$ be a statistic such that $f(\sigma_n)$ convergences for the uniform  law on $V^n=V^n_n$ and $\sum_          {i=1}^\infty \underline{\varepsilon}'_{n,i,\infty}(f)<\infty$. 
We obtain then  that  $f(\sigma_n)$ converges   for the Mallows law when its parameter goes to $\lambda$. More generally, it converges for any distribution such that $Class(i)$ converges in probability to infinity. 

}

\subsection*{Acknowledgements}
The author would like to acknowledge many extremely useful conversations
with  Myl\`{e}ne Ma\"{\i}da and Adrien Hardy  and their great help to improve the coherence of some parts of this paper. He would also  acknowledge  useful discussions  with 
Valentin Féray and  Natasha Blitvić about permutation patterns.

\appendix

\section{Ewens measures}
\label{apn1}
\begin{definition} \label{Ewens} 
Let $\theta$ be a non-negative real number. We say that a random \label{def:ew} permutation $\sigma_{Ew,\theta,n}$  follows the Ewens distribution with parameter $\theta$ if for all $\sigma \in \mathfrak{S}_n$,
\begin{align}\label{defEwens}
\mathbb{P}(\sigma_{Ew,\theta,n}=\sigma)= \frac{\theta^{\#(\sigma)-1}}{\prod_{i=1}^{n-1}(\theta+i)}.
\end{align}
\end{definition} 
Note that when $\theta=1$, the Ewens distribution is just the uniform distribution on $\mathfrak{S}_n$, whereas when $\theta=0$ we have the uniform distribution on permutations having a unique cycle. For general $\theta$, the Ewens distribution is clearly   conjugation invariant since it only involves the cycle structure of $\theta$.
\paragraph{} We want to recall the  interpretation of the Ewens distribution via a nice stochastic process known as "the Chinese restaurant process". 
Suppose that  there are  an infinite number of circular tables with infinite capacity.\begin{itemize}
    \item At $t=0$, all tables are empty.
    \item At $t=1$, the person $"1"$ comes and sits in the first table.
    
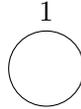
\begin{center}
\begin{tikzpicture}
\draw (0,0) circle (.5); 
\draw (0,.5) node[above]{$1$} ;
\end{tikzpicture}
\end{center}

    \item At $t=2$, the person $"2"$ comes and sits in the table near person $1$ with probability $\frac{1}{1+\theta}$ 
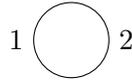
 \begin{center}
\begin{tikzpicture}
\draw (0,0) circle (.5); 
\draw (.5,0) node[right]{$2$} ;
\draw (-.5,0) node[left]{$1$} ;
\end{tikzpicture}
\end{center}
   
    and sits alone in a new table with probability $\frac \theta{1+\theta}$.
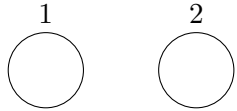
    \begin{center}
\begin{tikzpicture}
\draw (0,0) circle (.5); 
\draw (0,.5) node[above]{$1$} ;
\draw (2,0) circle (.5);
\draw (2,.5) node[above]{$2$} ;
\end{tikzpicture}
\end{center}

    \item At $t=n$, the person $"n"$ comes, she/he chooses to sit alone in a new table with probability $\frac{\theta}{\theta+n-1}$ and in an occupied table $i$ with probability $\frac{|B_i|}{\theta+n-1}$, where $B_i$ is the number of persons at the table $i$. In this case, she/he chooses her/his position uniformly in gaps between two persons.  
\end{itemize}   
For example, if we have the following configuration\footnote{we omitted empty tables},
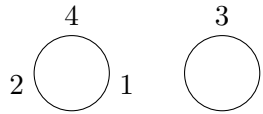
\begin{center}
\begin{tikzpicture}
\draw (0,0) circle (.5); 
\draw (.5,-0.15) node[right]{$1$} ;
\draw (0,.5) node[above]{$4$} ;
\draw (-0.5,-0.15) node[left]{$2$} ;
\draw (2,0) circle (.5);
\draw (2,.5) node[above]{$3$} ;
\end{tikzpicture}
\end{center}
at $t=5$,  the probability to switch to each of the following configurations  
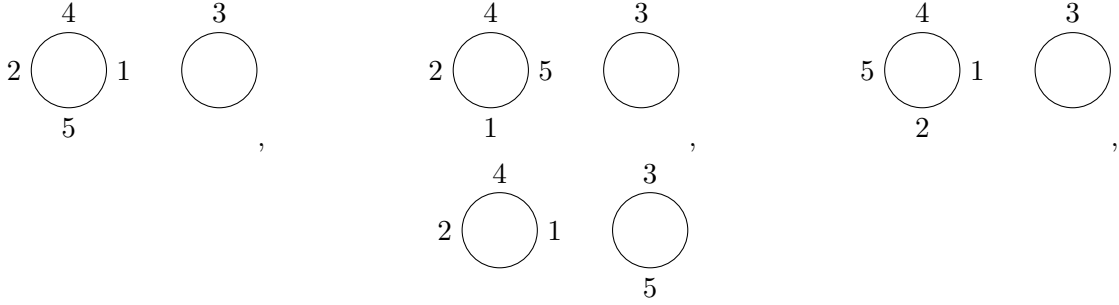
\begin{center}
\begin{tikzpicture}
\draw (0,0) circle (.5); 
\draw (.5,-0) node[right]{$1$} ;
\draw (0,.5) node[above]{$4$} ;
\draw (0,-.5) node[below]{$5$} ;
\draw (-0.5,0) node[left]{$2$} ;
\draw (2,0) circle (.5);
\draw (2,.5) node[above]{$3$} ;
\end{tikzpicture}, \quad \quad \quad \quad \quad
 \begin{tikzpicture}
\draw (0,0) circle (.5); 
\draw (.5,0) node[right]{$5$} ;
\draw (0,.5) node[above]{$4$} ;
\draw (0,-.5) node[below]{$1$} ;

\draw (-0.5,0) node[left]{$2$} ;
\draw (2,0) circle (.5);
\draw (2,.5) node[above]{$3$} ;\end{tikzpicture} 
, \quad \quad \quad \quad \quad
 \begin{tikzpicture}
\draw (0,0) circle (.5); 
\draw (.5,0) node[right]{$1$} ;
\draw (0,.5) node[above]{$4$} ;
\draw (0,-.5) node[below]{$2$} ;

\draw (-0.5,0) node[left]{$5$} ;
\draw (2,0) circle (.5);
\draw (2,.5) node[above]{$3$} ;\end{tikzpicture}, \quad \quad 
 \begin{tikzpicture}
\draw (0,0) circle (.5); 
\draw (.5,0) node[right]{$1$} ;
\draw (0,.5) node[above]{$4$} ;
\draw (-0.5,0) node[left]{$2$} ;
\draw (2,0) circle (.5);
\draw (2,-.5) node[below]{$5$} ;

\draw (2,.5) node[above]{$3$} ;\end{tikzpicture}
\end{center}
is $\frac{1}{\theta+4}$ 
 and the probability to switch to 
 
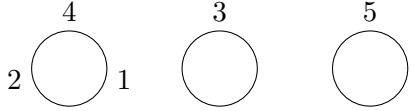
\begin{center}
\begin{tikzpicture}
\draw (0,0) circle (.5); 
\draw (.5,-0.15) node[right]{$1$} ;
\draw (0,.5) node[above]{$4$} ;
\draw (-0.5,-0.15) node[left]{$2$} ;
\draw (2,0) circle (.5);
\draw (2,.5) node[above]{$3$} ;
\draw (4,0) circle (.5);
\draw (4,.5) node[above]{$5$} ;
\end{tikzpicture}
\end{center}
is equal to $\frac{\theta}{4+\theta}$.

To obtain the  associated  permutation  to a configuration one reads the elements on each non-empty circle  counterclockwise to get a cycle. For example, to the configuration  
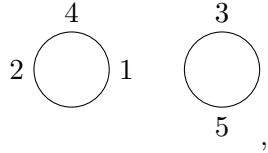
\begin{center}
\begin{tikzpicture}
\draw (0,0) circle (.5); 
\draw (.5,0) node[right]{$1$} ;
\draw (0,.5) node[above]{$4$} ;
\draw (-0.5,0) node[left]{$2$} ;
\draw (2,0) circle (.5);
\draw (2,-.5) node[below]{$5$} ;
\draw (2,.5) node[above]{$3$} ;\end{tikzpicture},
\end{center}
we associate the permutation $(1,4,2)(3,5)$.
\paragraph{} Using the Chinese restaurant process description of the Ewens distribution, it is obvious to see that the number of cycles
$\#(\sigma_{Ew,\theta,n})$ is the sum of $n$ independent Bernoulli random variables with parameters $\left\{\frac{\theta}{\theta+i}\right\}_{0\leq i \leq n-1}$.
For further reading, we recommend \citep{aldous,McCullagh2011,00806514}. 
In particular, we have the following classic result.
\begin{proposition}\label{numb_cyc_unif}
\begin{align*}
&\E(\#_1(\sigma_{Ew,\theta,n}))=\frac{n\theta}{n-1+\theta}
\quad  \text{and} \quad\E(\#(\sigma_{Ew,\theta,n}))=1+\sum_{i=2}^n \frac{\theta}{i-1+\theta}\leq 2+\theta log(n).
\end{align*}
\end{proposition}

 In particular, for uniform distribution, we have
 
 \begin{align}
     \frac{\#(\sigma_{unif,n})}{log(n)} \xrightarrow[n\to\infty]{\mathbb{P}}1.
 \end{align}
 \begin{proof}
 Since the number of cycles of the uniform law is the sum of $n$ independent random Bernoulli variables of parameters 
 $1,\frac{1}{2},\dots,\frac{1}{n}$ 
 and using Chebyshev's inequality, we obtain
 $$ \mathbb{P}\left(\left|\frac{\#(\sigma_{unif,n})}{\log(n)}-1\right|>\alpha \right)\leq \frac{ \frac{\sum_{i=1}^n \frac{i-1}{i^2} }{\log(n)^2}   }{\left(\alpha+1- \frac{\sum_{i=1}^n \frac{1}{i} }{\log{n}}  \right)^2}=O\left(\frac{1}{\log(n)}\right). $$ 
 \end{proof}
 \begin{remark}
 This convergence holds almost surely. The proof uses martingale techniques.  One can find a proof of this result in \citep{00806514}.
 \end{remark}
 One can now apply our results using the following two results. 
 \begin{corollary} 
Let $(\theta_n)_{n\geq 1}$ be a sequence of  non-negative real numbers
such that:
\begin{equation}
\lim_{n\to \infty} \frac{\theta_n \log(n)  }{n^\frac 1 \alpha}=0.
\end{equation}
Then $(\sigma_{Ew,\theta_n,n})_{n\geq 1}$ satisfies  \eqref{hinv1}.
\end{corollary}

 \begin{proposition} \label{lpewens}
For any {$\theta \geq 0,\alpha>0$ and  $p\geq [1,\infty[$} , $(\sigma_{Ew,\theta,n})_{n\geq 1}$ satisfies  \eqref{hinv1p}.
 \end{proposition}
 \begin{proof}
 Using Bernstein inequality, if $\theta \geq 1$,
 \begin{align*}
 \mathbb{P}(\#(\sigma_{Ew,\theta,n})> (3p+1)\theta\log(n) +2 )&\leq \mathbb{P}(\#(\sigma_{Ew,\theta,n})> \mathbb{E}(\ \#(\sigma_{Ew,\theta,n})) + 3p\theta\log(n))  \\ &\leq \exp\left(\frac{-\frac{9}{2}\theta^2p^2 \log(n)^2} {\mathrm{var}(\#(\sigma_{Ew,\theta,n}))+ \frac{3p}{3}\theta  \log(n)  }\right) 
 \\ &\leq \exp\left(\frac{-\frac{9}{2}\theta^2p^2 \log(n)^2} { (p+1)\theta\log(n)+2   }\right) = O\left( n^{-\frac{9}{4}\theta \frac{p^2}{p+1}}\right)= O\left( n^{-\frac{9p}{8}}\right).
 \end{align*}
 Consequently,
 $$\mathbb{E} (\#(\sigma_{Ew,\theta,n})^p ) \leq  ((3p+1)\theta\log(n) +2 )^p   + n^p O(  n^{-\frac{9p}{8}}) = O(\log^p(n)).  $$
When $\theta<1$, one can conclude since $\mathbb{E} (\#(\sigma_{Ew,\theta,n})^p )<\mathbb{E} (\#(\sigma_{Ew,1,n})^p ).$ 
\end{proof}

\end{document}